\documentclass[a4paper,10pt,reqno]{amsart}

\usepackage{enumerate}
\usepackage{amssymb,amsmath,amsthm,latexsym,mathrsfs}
\usepackage{amsbsy,amsfonts,mathtools,slashed}
\usepackage{bm}
\usepackage{graphicx,color,yfonts}
\usepackage[numbers,sort]{natbib}

\usepackage[colorlinks=true]{hyperref}
\hypersetup{linkcolor=red,citecolor=blue,filecolor=dullmagenta,urlcolor=blue} 

\newtheorem{thm}{Theorem}[section]
\newtheorem{lemma}[thm]{Lemma}
\newtheorem{prop}[thm]{Proposition}
\newtheorem{cor}[thm]{Corollary}

\newtheorem{rem}[thm]{Remark}

\makeatother
\numberwithin{equation}{section}
\numberwithin{figure}{section}
\theoremstyle{plain}
\theoremstyle{plain}
\theoremstyle{plain}

\makeatother
\numberwithin{equation}{section}
\numberwithin{figure}{section}
\theoremstyle{plain}
\theoremstyle{plain}
\theoremstyle{plain}

\newcommand{\les}{\lesssim}

\newcommand{\ve}{{\varepsilon}}

\newcommand{\R}{{\mathbb R}}

\newcommand{\T}{{\mathbb T}}

\def\cF{\mathcal F}

\def\cH{\mathcal H}
\def\cI{\mathcal I}

\def\cR{\mathcal R}

\newcommand{\p}{\partial}

\def\bra#1{\left\langle #1\right\rangle}

\def\wt#1{\widetilde{#1}}
\def\wh#1{\widehat{#1}}
\def\ol#1{\overline{#1}}

\newcommand{\halbe}{\mathcal{H}_{\alpha,\beta}}

\def\sgn#1{{\rm sgn}(#1)}

\begin{document}
\title[Local and global well-posedness for the kinetic derivative NLS on $\mathbb{R}$]{Local and global well-posedness for the kinetic derivative NLS on $\mathbb{R}$}

\dedicatory{{\it Dedicated to Professor Yoshio Tsutsumi on the occasion of his seventieth birthday}}

\author{Nobu Kishimoto}
\address{Research Institute for Mathematical Sciences(RIMS), Kyoto University, Kitashirakawa Oiwake-cho, Sakyo-ku, Kyoto, 606-8502, Japan}
\email{nobu@kurims.kyoto-u.ac.jp}

\author{Kiyeon Lee}
\address{Department of Mathematical Sciences, Korea Advanced Institute of Science and Technology, 291 Daehak-ro, Yuseong-gu, Daejeon, 34141, Republic of Korea}
\email{kiyeonlee@kaist.ac.kr}

\thanks{2020 {\it Mathematics Subject Classification.} 35Q55, 35A01, 35B30, 35B45.}
\thanks{{\it Keywords and phrases.} Kinetic derivative nonlinear Schr\"odinger equation, Well-posedness, Energy method, Gauge transformation, A priori estimate.}

\begin{abstract}
We investigate the local and global well-posedness of the kinetic derivative nonlinear Schr\"odinger equation(KDNLS) on $\R$, described by
\[ i\partial_t u + \partial_x^2 u = i\alpha \partial_x (|u|^2 u) + i\beta \partial_x (\cH(|u|^2) u), \]
where $\alpha, \beta \in \mathbb{R}$, and $\cH $ represents the Hilbert transformation. 
For KDNLS, the $L^2$ norm of a solution is decreasing (resp.\ increasing, conserved) when $\beta$ is negative (resp.\ positive, zero).
Focusing on the Sobolev spaces $H^2$ and $H^2 \cap H^{1,1}$, we establish local well-posedness via the energy method combined with gauge transformations to address resonant interactions in both cases of negative and positive $\beta$. 
For the dissipative case $\beta < 0$, we further demonstrate global well-posedness by deriving an a priori bound in $H^2$.
\end{abstract}

\maketitle

\tableofcontents


\section{Introduction}

We consider the following one-dimensional nonlinear Schr\"odinger equation with local and non-local cubic derivative nonlinear terms:
\begin{equation}
	\left\{\begin{aligned}
		i\partial_tu + \partial_x^2u &= i\alpha \partial_x \big[ |u|^2u\big] +i\beta \partial_x\big[\mathcal{H}(|u|^2)u\big] , \qquad t>0, ~ x \in \mathbb{R},\\
		u(0,x) &= \phi(x),
	\end{aligned}\right.\label{kdnls}
\end{equation}
where $\alpha ,\beta \in \mathbb{R}$ are constants and $\cH$ is the Hilbert transformation defined by
\[ \mathcal{H}f(x) :=\frac{1}{\pi} p.v.\int_{-\infty}^\infty\frac{f(y)}{x-y}\,dy = \mathcal{F}^{-1}\left[ -i\, \mathrm{sgn}(\xi)\mathcal{F}f\right] (x). \]
We may focus on the positive time direction by considering time reflection $u(t,x)\mapsto \ol{u(-t,-x)}$, which leaves the equation invariant except that $\beta$ needs to be replaced with $-\beta$.
 
The equation \eqref{kdnls} with $\beta \neq 0$ was derived in \cite{MW86,MW88} as a model for the propagation of weakly nonlinear and weakly dispersive Alfv\'en waves in a collision-less plasma.
The non-local nonlinear term $\p_x[\cH(|u|^2)u]$ representing the effect of resonant particles is obtained using a kinetic model, and thus we refer to \eqref{kdnls} with $\beta \neq 0$ as the \emph{kinetic derivative nonlinear Schr\"odinger equation} (KDNLS).
Formally the $L^2$ norm of a solution $u$ to \eqref{kdnls} satisfies  the following identity:
\[ \frac{d}{dt}\| u(t)\|_{L^2}^2\ =\ \beta \big\| D_x^{\frac12}(|u(t)|^2)\big\|_{L^2}^2, \]
where $D_x:=|\partial_x|=\mathcal{H}\partial_x$.
In particular, the equation becomes dissipative when $\beta <0$. 
In the physical model, the constant $\beta$ is determined depending on the velocity distributions of particle species present, and both cases of negative and positive $\beta$ appear to be physically relevant.
It is known that $\beta$ becomes negative if all the velocity distributions decrease as functions of the parallel velocity component to the magnetic field at the resonant velocity (see \cite{MW86,MW88} and discussion in \cite{PdMP11}).

We recall previous results on the well-posedness of the Cauchy problem \eqref{kdnls} for KDNLS in Sobolev spaces.
Note that KDNLS is invariant under the scaling transformation
\[ u(t,x)~\mapsto ~\lambda^{\frac12}u(\lambda^2t,\lambda x)\qquad (\lambda >0),\]
which gives the scale-critical space $L^2$.
The case $\beta =0$ (i.e., the standard derivative NLS) has been very well studied; for a survey of the previous results in this case, we refer to the recent work by Harrop-Griffiths, \mbox{Killip}, \mbox{Ntekoume}, and Vi\c{s}an \cite{HGKNV24}, who obtained the global well-posedness in $L^2$ using complete integrability.
Compared to the case $\beta=0$, many techniques are not equally applicable and much less results are known for KDNLS.
Even the construction of (local) smooth solutions by the classical energy method is not straightforward and seems to require dissipative sign $\beta <0$ (see \cite[Appendix]{KT-gauge} for a proof in $H^s$, $s>3/2$).
Rial~\cite{R02} constructed global $L^2$-weak solutions to \eqref{kdnls} with $\beta <0$ and showed smoothing effect by using dissipative structure.
Peres~de~Moura and Pastor~\cite{PdMP11} made an important observation on the applicability of Kato smoothing and maximal function estimates to the non-local nonlinearity.%
\footnote{The authors in \cite{PdMP11} actually claimed small-data local well-posedness in $H^s(\R)$ for $s>1/2$ with \emph{both signs of $\beta$}.
However, as pointed out in \cite{KT23-1}, their proof seems to have several gaps, especially in the estimates of an anti-derivative term in \cite[Eq.~(3.25)]{PdMP11}.
We do not know how to fix it for now.}
Recently, the first author and Tsutsumi~\cite{KT-stFr} proved local well-posedness in $H^s(\R)$ for $s>1$ under the presence of dissipation $\beta <0$ by using short-time $U^2$ and $V^2$ spaces.
In \cite{KT23-1} they also revealed that the nature of the equation changes dramatically under periodic boundary conditions: the dissipative structure makes the equation of parabolic type through the resonant frequency interactions.
Using the parabolic smoothing effect for periodic domain $\T$, they proved global well-posedness in $H^s(\mathbb{T})$, $s>1/4$ in the case $\beta<0$~\cite{KT-gauge} and non-existence of solutions for $s>3/2$ in the case $\beta >0$~\cite{KT23-1}.
See \cite[Introduction]{KT-gauge} for more discussion on comparison of derivative NLS and KDNLS, methods of proof of well-posedness, and related results.

In this paper, we aim to address the local and global well-posedness of \eqref{kdnls} in $H^2$ and in $H^2\cap H^{1,1}$.
Here, we define the weighted Sobolev space $H^{s,m}$ by the norm
\[ \| f\|_{H^{s,m}}:=\| \bra{x}^m\bra{\p_x}^sf\|_{L^2}, \]
where 
\[ \bra{x}:=(1+|x|^2)^{\frac12},\quad \bra{\p_x}^s:=\mathcal{F}^{-1}\bra{\xi}^s\mathcal{F}.\]
Now, we introduce the first theorem.
\begin{thm}[Local well-posedness in $H^2$]\label{thm:lwp}
Let $\alpha ,\beta \in \R$, then the Cauchy problem \eqref{kdnls} is locally well-posed in $H^2(\mathbb{R})$. More precisely, the following hold:

(i) For any $\phi\in H^2$, there exist $T=T(\| \phi\|_{H^2})>0$ and a unique solution $u\in C([0,T],H^2)$ of \eqref{kdnls} on $[0,T]$.
Moreover, the solution map $\phi\mapsto u$ from $H^2$ to $C([0,T],H^2)$ is continuous.

(ii) If in addition $\phi\in H^2\cap H^{1,1}$, then $u\in C([0,T],H^2\cap H^{1,1})$ with the same $T$.
Moreover, the solution map is continuous as a map from $H^2\cap H^{1,1}$ to $C([0,T],H^2\cap H^{1,1})$.

(iii) Let $X^k$ be either $H^k$ or $H^k\cap H^{k-1,1}$ for any integer $k\geq 3$.
If in addition $\phi \in X^k$, then $u\in L^\infty([0,T],X^k)\cap C([0,T],X^{k-1})$ with the same $T$.
\end{thm}

To the best of our knowledge, Theorem~\ref{thm:lwp} is the first result on (large-data) well-posedness for KDNLS with $\beta >0$.
It is noteworthy that, for $\beta >0$, the forward-in-time well-posedness fails on the torus in any Sobolev spaces due to non-existence of solutions (see \cite{KT23-1}). 
In contrast, our result establishes local well-posedness even for $\beta>0$, showing that the ill-posedness phenomena are directly related not only to the sign of $\beta$, but also to the geometry of the underlying domain.

We remark that in (ii) we prove not only the persistence of $H^{1,1}$ but also the continuity of the solution map in the $H^2\cap H^{1,1}$ topology.
Therefore, (i) and (ii) imply the local well-posedness in $H^2\cap H^{1,1}$, which we will use in our next work \cite{KL-scat} (see Remark~\ref{rem:thm} below).
On the other hand, the persistence of regularity property given in (iii) is weaker than the case of $H^2\cap H^{1,1}$ in (ii).
In particular, it does not imply local well-posedness in any smaller space $X^k$ ($k\geq 3$) in the sense that the strong continuity of the solution in $X^k$ and the continuity of the solution map in the $X^k$ topology are missing.
Although further consideration would reveal these properties and yield local well-posedness in $X^k$, we will not address it in the present paper to avoid overcomplicating the proof.
Note that the weak statement in (iii) will be enough for our application in \cite{KL-scat}.

To prove Theorem~\ref{thm:lwp}, we employ the energy method used in the  dissipative case $\beta <0$ (see \cite[Appendix]{KT-gauge}) but combine it with the gauge transformation technique developed for the periodic problem, to deal with the case $\beta>0$.
Our proof begins with proving a priori estimates for regularized equations of \eqref{kdnls}. 
The most delicate resonance in the energy estimates arises from the nonlinear term $\cH(\ol{u}\p_x u)u$.
As observed in the previous result for periodic domain \cite{KT-gauge}, this type of resonance can be effectively removed by introducing frequency-restricted gauge transformations.
While the method of gauge transformation has been previously used in the context of energy estimates for derivative NLS with general nonlinearities (see, e.g., \cite{C94}), our gauge transformations are different in that they need to be defined separately for positive- and negative-frequency parts of the solution.
We also note that simpler gauge transformations compared to those in \cite{KT-gauge} turn out to be sufficient for the energy bounds.
Then, thanks to the gauge transformations, this allows us to use the normal form approach (or in other words, modified energy method) appropriately in the energy estimates to deal with the derivative losses. 
The processes for $H^2$ bounds and weighted bounds are not significantly different; however, the commutative relations between the operator 
\[ J(t):= x +2it\p_x \]
and other operators are utilized to handle terms with weight more carefully. 
After we prove the a priori bounds for regularized equations, we prove difference estimates in $H^2$ and $H^{1,1}$, respectively. 
These estimates allow us to prove Theorem \ref{thm:lwp}.

The second aim of this paper focuses on the global well-posedness. 
In the following theorem, we extend our well-posedness result in $H^2$ globally in the dissipative case $\beta<0$.
\begin{thm}[Global well-posedness in $H^2$]\label{thm:gwp}
	Let $\alpha \in \mathbb{R}$ and $\beta<0$. Then, the Cauchy problem \eqref{kdnls} is (forward in time) globally well-posed in $H^2(\mathbb{R})$.
\end{thm}

As far as we know, this is the first global well-posedness result for KDNLS posed on $\mathbb{R}$.
We prove a global a priori bound in $H^2$ in Proposition \ref{prop:H2bound} below, and this and the local well-posedness in Theorem~\ref{thm:lwp} immediately imply Theorem \ref{thm:gwp}.
Clearly, the same global result holds in $H^2(\mathbb{R})\cap H^{1,1}(\mathbb{R})$.

\begin{rem}\label{rem:thm}
	Our global well-posedness result naturally leads to the investigation of the asymptotic behavior of solutions to \eqref{kdnls} on $\R$. In our next work \cite{KL-scat}, we consider the asymptotic behavior of small solutions to \eqref{kdnls} with $\beta \neq 0$ in $H^2 \cap H^{1,1}$ (especially both signs of $\beta$). 
The (local) well-posedness result of this paper provides a rigorous justification for various formal arguments in analyzing the asymptotic behavior.
\end{rem}


This paper is organized as follows.
In Section~\ref{sec:gauge}, we introduce a regularized equation and frequency-restricted gauge transformations for KDNLS.
In Section~\ref{sec:useful}, we collect some useful lemmas.
Various a priori estimates for smooth solutions to the regularized equation are shown in Section~\ref{sec:apriori}, and those for the difference of solutions are given in Section~\ref{sec:diff}.
Once these estimates are available, Theorem~\ref{thm:lwp} can be shown by a standard argument; we give a proof in Appendix~\ref{sec:proof-lwp}.
Finally, a global-in-time $H^2$ a priori bound is stated and proved in Section~\ref{sec:globalh2}, which together with Theorem~\ref{thm:lwp} implies Theorem~\ref{thm:gwp}.


\subsection*{Notations}
\begin{itemize}
\item
(Mixed-normed spaces) 
For $T>0$ and a Banach space $X$, we denote $C_TX:=C([0,T],X)$ and $\| u\| _{L^\infty_TX}:=\| \| u(t)\|_X \|_{L^\infty([0,T])}$.

\item 
As usual different positive constants are denoted by the same letter $C$, if not specified.
$A\lesssim B$ and $A\gtrsim B$ mean that $A\le CB$ and $A\ge C^{-1}B$,
respectively for some $C>0$. $A\sim B$ means that $A\lesssim B$
and $A\gtrsim B$. 
We also write as $A\les _MB$ to emphasize the dependence of the implicit constant on the quantity $M$.

\item 
We fix our definition of the (one-dimensional) Fourier and inverse Fourier transforms as
\[ \cF_x[f](\xi )=\wh{f}(\xi) :=\frac{1}{\sqrt{2\pi}}\int_\R e^{-ix\xi}f(x)dx \quad \mbox{and}\quad \cF^{-1}_\xi [g](x):=\frac{1}{\sqrt{2\pi}}\int_\R e^{ix\xi}g(\xi)d\xi . \]

\item
(Littlewood-Paley operators) 
Let $\varrho$ be a Littlewood-Paley function such that $\varrho\in C_{0}^{\infty}(\{ |\xi|<2\})$ with $\varrho(\xi)=1$ for $|\xi|\le1$ and define $\varrho_{N}(\xi):=\varrho\big(\frac{\xi}{N}\big)-\varrho\big(\frac{2\xi}{N}\big)$ for $N\in2^{\mathbb{Z}}$. 
We write $\varrho_{\le N_{0}}:=\sum_{N\leq N_{0}}\varrho_{N}=\varrho\big( \frac{\xi}{N_0}\big)$ and $\varrho_{>N_0}:=1-\varrho_{\leq N_0}$.
Then we define the (spatial) frequency projection $P_{N}$ by $\mathcal{F}(P_{N}f)(\xi)=\varrho_{N}(\xi)\widehat{f}(\xi)$. 
In addition, we write $P_{\leq N_0}:=\sum_{N\leq N_0}P_N$ and $P_{>N_0}:=\sum _{N>N_0}P_N$.
\end{itemize}


\section{Regularized equation and gauge transformation}
\label{sec:gauge}

We write the nonlinear terms in \eqref{kdnls} as $i\p_x\mathcal{N}[u]$, where
\[ \mathcal{N}[u]:=\alpha |u|^2u +\beta \cH(|u|^2)u.\]
We also introduce the notation 
\[ \halbe :=\alpha \,{\rm Id}+\beta\, \cH ,\]
so that $\mathcal{N}[u]=\halbe (|u|^2)u$.

\subsection{Regularized equations}

To prove Theorem~\ref{thm:lwp}, we consider the regularized equation with smooth initial data. Indeed, for $\ve \in(0,1)$, we consider
\begin{align}\label{kdnls-ve}
	\left\{ \begin{aligned}
		\p_tu^\ve -i\p_x^2u^\ve &= \ve \p_x^2u^\ve +\p_x \mathcal{N}[u^\ve ] ,\\
		u^\ve (0,x)&=\phi^\ve (x).
	\end{aligned}\right.
\end{align}
We also consider the difference of two solutions to \eqref{kdnls-ve} with possibly different $\ve$'s and $\phi^{\ve}$'s.
Let $ 0<\ve_2 \le \ve_1 <1$ and for $j=1,2$, let $u^{\ve_j}$
be a solution to \eqref{kdnls-ve} with $\ve=\ve_j$ and $\phi^\ve=\phi^{\ve_j}$.
Set
\[ w:=u^{\ve_1}-u^{\ve_2}. \]
Then the equation for $w$ is given below:
\begin{align}\label{eq:w}
\left\{ \begin{aligned}
\p_tw-i\p_x^2w&=\ve_2\p_x^2w+(\ve_1-\ve_2)\p_x^2u^{\ve_1}+\p_x\wt{\mathcal{N}}[u^{\ve_1},w],\\
w(0,x)&=\phi^{\ve_1}(x)-\phi^{\ve_2}(x),
\end{aligned}\right.
\end{align}
where
\[ \wt{\mathcal{N}}[u^{\ve_1},w]:=\mathcal{N}[u^{\ve_1}]-\mathcal{N}[u^{\ve_1}-w]=\halbe(|u^{\ve_1}|^2)u^{\ve_1}-\halbe(|u^{\ve_1}-w|^2)(u^{\ve_1}-w).\]

We also consider the equation for $J(t)u^\ve(t)$, $J(t):=x+2it\p_x$.
Recall that the operator $J(t)$ satisfies the commutator relations
\[ [J(t),\p_t-i\p_x^2]=0,\quad [J(t),\p_x]=[x,\p_x]=-1,\quad [J(t),\p_x^2]=[x,\p_x^2]=-2\p_x,\]
and the ``Leibniz rule''
\[ J[fg]=2it(\p_xf)g+fJg,\quad J\big[ \halbe(f\ol{g})h\big] =\halbe((Jf)\ol{g})h-\halbe(f\ol{Jg})h+\halbe (f\ol{g})Jh.\]
For a solution $u^\ve$ to \eqref{kdnls-ve}, a direct calculation shows that $Ju^\ve$ satisfies the following equation:
\begin{equation}\label{eq:Ju}
\p_tJu^\ve -i\p_x^2Ju^\ve = \ve \p_x^2Ju^\ve +\p_xJ\mathcal{N}[u^\ve] -2\ve \p_xu^\ve -\mathcal{N}[u^\ve].
\end{equation}
Similarly, the equation for $Jw$ is given by
\begin{equation}\label{eq:Jw}
\begin{aligned}
\p_tJw-i\p_x^2Jw&= \ve_2\p_x^2Jw+(\ve_1-\ve_2)\p_x^2Ju^{\ve_1}+\p_xJ\wt{\mathcal{N}}[u^{\ve_1},w]\\
&\quad -2\ve_2\p_xw-2(\ve_1-\ve_2)\p_xu^{\ve_1}-\wt{\mathcal{N}}[u^{\ve_1},w].
\end{aligned}
\end{equation}

\subsection{Gauge transformation}
Now, we introduce a gauge transformation which removes the most problematic nonlinear interactions of the high-low type; i.e., interactions between high frequency components of the differentiated function and much lower frequency components of the other two.
The gauge transformation for \eqref{kdnls} has been investigated in \cite{KT-gauge} to treat the nonlinear terms $2\alpha |u|^2\partial_xu$, $\beta \mathcal{H}(|u|^2)\partial_xu$ and $\beta \mathcal{H}(\ol{u}\partial_xu)u$. 
In this work we only remove the high-low interaction from $\beta \mathcal{H}(\ol{u}\partial_xu)u$, as the other two terms exhibit good symmetry so that we can apply integration by parts in the energy estimates.
To this end, we define 
\[ \cF(P_\pm f)(\xi) := \chi_\pm(\xi) \wh{f}(\xi),\qquad \cF(Q_\pm f)(\xi ):=\chi_\pm(\xi) \varrho_{>1}(\xi) \wh{f}(\xi),
\]
where $\chi_\pm$ is the characteristic function of $\pm (0,\infty)$.
Note that $\mathcal{H}=-iP_++iP_-$ and $Q_\pm \mathcal{H} =\mp iQ_\pm$.
Hence, if we fix the sign of output frequencies by applying $Q_\pm$, then the high-low interactions of $Q_\pm \big[ \beta \mathcal{H}(\ol{u}\partial_xu)u\big]$ are the same as those of $\beta \ol{u}(Q_\pm \mathcal{H}\partial_xu)u=\mp i\beta |u|^2\p_xQ_\pm u$.
This suggests (see, e.g., \cite[Section~2]{KT94}) that the suitable gauge function for the equation of $Q_\pm u$ is given by $e^{\rho ^\pm [u]}$, where 
\begin{equation}\label{def:gauge}
\rho ^\pm [u]:=\mp \frac\beta2 \int_{-\infty}^x |u(y)|^2\,dy.
\end{equation}
For the difference equation \eqref{eq:w}, the nonlinear terms can be written as
\begin{align}\label{eq:wtN}
\begin{aligned}
\p_x\wt{\mathcal{N}}[u^{\ve_1},w]
&=\halbe\big(\p_xw\,(\ol{u^{\ve_1}-w})\big) (u^{\ve_1}-w)+\halbe\big( (u^{\ve_1}-w)\p_x\ol{w}\big) (u^{\ve_1}-w)\\
&\qquad +\halbe(|u^{\ve_1}-w|^2) \p_xw +\wt{\cR} \\
&=\halbe(\p_xw\, \ol{u^{\ve_2}})u^{\ve_2}+\halbe(u^{\ve_2}\p_x\ol{w})u^{\ve_2}+\halbe(|u^{\ve_2}|^2)\p_xw +\wt{\cR},
\end{aligned}
\end{align}
where $\wt{\cR}$ is the sum of the terms in which $\p_x$ hits one of $u^{\ve_1}$'s.
In the energy estimates for $w$, it will turn out that only the term $\beta \cH(\p_xw\, \ol{u^{\ve_2}})u^{\ve_2}$ is problematic.
Since the high-low interaction part of $\beta Q_\pm \big[ \cH(\p_xw\, \ol{u^{\ve_2}})u^{\ve_2}\big]$ is identical with that of $\beta |u^{\ve_2}|^2Q_\pm \cH\p_xw=\mp i\beta |u^{\ve_2}|^2\p_xQ_\pm w$, we will use the gauge functions $e^{\rho^\pm[u^{\ve_2}]}$ for $Q_\pm w$.

\begin{rem}
In the dissipative case $\beta <0$, the energy method works without the gauge transformation (see \cite[Appendix]{KT-gauge}).
For instance, in the $L^2$ energy estimate for the difference $w$, we can treat the term $\beta \p_x\cH(2{\rm Re}(u^{\ve_2}\ol{w}))u^{\ve_2}$ including the problematic case $\beta \cH(\p_xw\,\ol{u^{\ve_2}})u^{\ve_2}$, as follows:
\[ 2{\rm Re} \int \beta \p_x\cH\big[ 2{\rm Re} (u^{\ve_2}\ol{w})\big] u^{\ve_2}\ol{w}\,dx=\beta \Big\| D_x^{1/2}\big[ 2{\rm Re}(u^{\ve_2}\ol{w})\big]\Big\|_{L^2}^2\leq 0. \]
In the case $\beta >0$, however, this term is not estimated by $C\| u^{\ve_2}\|_{H^2}^2\| w\|_{L^2}^2$.
We thus need to exploit the gauge transformations to suitably handle the problematic term and refine the estimate.
\end{rem}


\section{Useful lemmas}
\label{sec:useful}

In this section, we present several estimates that will be used in the subsequent proof of a priori bounds.
\begin{lemma}\label{lem:commutator}
We have
\[ \big\| [Q_\pm,f]\p_xg\big\|_{L^2}\les \| f\|_{H^2}\| g\|_{L^2}. \]
\end{lemma}
\begin{proof}
The claimed estimate is equivalent to 
\[ \left\| \int_{\R}\big( (\chi_\pm \varrho_{>1})(\xi)-(\chi_\pm \varrho_{>1})(\eta)\big) \eta \wh{f}(\xi-\eta)\wh{g}(\eta)\,d\eta \right\|_{L^2_\xi}\les \| \bra{\xi}^2\wh{f}\|_{L^2_\xi}\| \wh{g}\|_{L^2_\xi}.\]
Since the condition $\bra{\eta}\gg \bra{\xi-\eta}$ implies that $(\chi_\pm \varrho_{>1})(\xi)-(\chi_\pm \varrho_{>1})(\eta)=0$, we have
\[ \big| \big( (\chi_\pm \varrho_{>1})(\xi)-(\chi_\pm \varrho_{>1})(\eta) \big)\eta\big| \les \bra{\xi-\eta}. \]
This inequality, Young's and H\"older's inequalities imply the claim.
\end{proof}

\begin{lemma}\label{lem:rho}
	For $u\in H^2$, we have $e^{\rho^\pm[u]}\in W^{2,\infty}$ and
\begin{gather*}
	\| e^{\rho^\pm[u]}\|_{L^\infty}\leq e^{C\| u\|_{L^2}^2},\qquad \| \p_xe^{\rho^\pm[u]}\|_{L^\infty}\les e^{C\| u\|_{L^2}^2}\| u\|_{H^1}^2,\\
	\| \p_x^2e^{\rho^\pm[u]}\|_{L^\infty}\les e^{C\| u\|_{L^2}^2}\big( \| u\|_{H^1}\| u\| _{H^2}+\| u\|_{H^1}^4\big) .
\end{gather*}
Moreover, if $u^\ve$ is a solution to \eqref{kdnls-ve}, we have
\[ \| \p_te^{2\rho^\pm[u^\ve(t)]}\|_{L^\infty}\les e^{C\| u^\ve\|_{L^2}^2}\big( \| u^\ve\|_{H^1}\| u^\ve\|_{H^2}+\| u^\ve\|_{H^1}^4\big) .\]
\end{lemma}
\begin{proof}
The desired bounds follow from direct calculations and the H\"older, Sobolev inequalities.
For the last bound, we use the identity
	\begin{align*}
		\p_t\rho ^\pm[u^\ve] 
		&=\mp \frac\beta2\int_{-\infty}^x 2\mathrm{Re}\, \Big[  \Big(i\p_x^2u^\ve+\ve \p_x^2 u^\ve +\alpha \partial_x \big[ |u^\ve|^2u^\ve\big] +\beta \partial _x\big[ \mathcal{H}(|u^\ve|^2)u^\ve\big] \Big) \ol{u^\ve}\Big]\,dy \\
		&=\mp \frac\beta2\int_{-\infty}^x\Big[ \ve \p_x^2(|u^\ve|^2) + \partial _x\Big(2\mathrm{Re}\,(i\ol{u^\ve}\p_x u^\ve) +\frac{3}{2}\alpha |u^\ve|^4+2\beta \mathcal{H}(|u^\ve|^2)|u^\ve|^2\Big) \\
			&\hspace{6cm} -2\ve |\p_x u^\ve |^2 -\beta \mathcal{H}(|u^\ve|^2)\p_x(|u^\ve|^2)\Big]\,dy \\
		&=\mp \frac\beta2\Big[ \ve \p_x(|u^\ve|^2) + 2\mathrm{Re}\,(i\ol{u^\ve}\p_x u^\ve) +\frac{3}{2}\alpha |u^\ve|^4+2\beta \mathcal{H}(|u^\ve|^2)|u^\ve|^2 \\
			&\hspace{4cm} -2\ve \int_{-\infty}^x|\p_x u^\ve |^2\,dy -\beta \int_{-\infty}^x\mathcal{H}(|u^\ve|^2)\p_x(|u^\ve|^2)\,dy\Big] .
	\end{align*}
\end{proof}

\begin{lemma}\label{lem:dt-bound}
	Let $k\geq 2$ be an integer and $u$ be a solution to \eqref{kdnls-ve}. Then we have
\begin{align*}
	\|\p_t (e^{-it\partial_x^2}u(t))\|_{H^{k-1}} &\les \varepsilon \| \partial_x^{k+1}u(t)\|_{L^2}+\Big( 1+\| u(t)\|_{H^1}^2\Big) \|u(t)\|_{H^k},\\
	\|\p_t (e^{-it\partial_x^2}J(t)u(t))\|_{H^{k-2}} &\les \varepsilon \| \partial_x^{k}J(t)u(t)\|_{L^2}+\Big( 1+\| u(t)\|_{H^1}\| J(t)u(t)\|_{H^1}\Big) \|u(t)\|_{H^{k-1}}\\
&\quad +\| u(t)\|_{H^1}^2\Big( \| u(t)\|_{H^{k-2}}+\| J(t)u(t)\|_{H^{k-1}}\Big).
\end{align*}	
\end{lemma}
\begin{proof}
	Since $\p_te^{-it\p_x^2}=e^{-it\p_x^2}(\p_t-i\p_x^2)$, the desired estimates follow from the equations \eqref{kdnls-ve}, \eqref{eq:Ju} and the Sobolev inequalities.
\end{proof}

\begin{lemma}\label{lem:int}
Let $k\geq 2$ be an integer.
We have
\begin{align}
&\begin{aligned}
&\Big\| \p_x^{k+1}\mathcal{N}[u] -\big( 2\alpha |u|^2+\beta \cH(|u|^2)\big) \p_x^{k+1}u-\beta \cH \big( \p_x^{k+1}u\, \ol{u}\big) u -\halbe \big(u\p_x^{k+1}\ol{u}\big) u\Big\|_{L^2}\\
&\quad \les \| u\|_{H^1}\| u\|_{H^2}\| u\|_{H^k},
\end{aligned} \label{lem:int1} \\
&\begin{aligned}
&\Big\| \p_x^kJ\mathcal{N}[u] -\big( 2\alpha |u|^2+\beta \cH(|u|^2)\big) \p_x^kJu-\beta \cH \big( \p_x^kJu\,\ol{u}\big) u +\halbe \big(u\p_x^{k}\ol{Ju}\big) u\Big\|_{L^2}\\
&\quad \les \| u\|_{H^1}\big( \| u\|_{H^2}\| Ju\|_{H^{k-1}}+\| Ju\|_{H^1}\| u\|_{H^k}\big) .
\end{aligned} \label{lem:int2}
\end{align}
\end{lemma}
\begin{proof}
Considering on the Fourier side, we see that 
\[ \text{L.H.S. of \eqref{lem:int1}}\les \Big\| \int_{\xi_1+\xi_2+\xi_3=\xi,\,|\xi_1|\leq |\xi_2|\leq |\xi_3|}|\xi_2||\xi_3|^k|\wh{v_1}(\xi_1)\wh{v_2}(\xi_2)\wh{v_3}(\xi_3)|\,d\xi_1\,d\xi_2 \Big\|_{L^2_\xi},\]
where each of $v_1,v_2,v_3$ is either $u$ or $\ol{u}$.
By Young's and Cauchy-Schwarz's inequalities, this is bounded by
\[ \| \wh{v_1}\|_{L^1_\xi}\| |\xi |\wh{v_2}\|_{L^1_\xi}\| |\xi |^k\wh{v_3}\|_{L^2_\xi}\les \| u\|_{H^1}\| u\|_{H^2}\| u\|_{H^k},\]
which implies \eqref{lem:int1}.
Similarly, writing $w$ to denote either $Ju$ or $\ol{Ju}$, we have
\begin{align*}
\text{L.H.S. of \eqref{lem:int2}}
&\les \Big\| \int_{\xi_1+\xi_2+\xi_3=\xi,\,|\xi_1|\leq |\xi_2|}\Big( |\xi_2||\xi_3|^{k-1}+|\xi_2|^k \Big) |\wh{v_1}(\xi_1)\wh{v_2}(\xi_2)\wh{w}(\xi_3)|\,d\xi_1\,d\xi_2 \Big\|_{L^2_\xi}\\
&\les \| u\|_{H^1}\| u\|_{H^2}\| Ju\|_{H^{k-1}}+\| u\|_{H^1}\| u\|_{H^k}\| Ju\|_{H^1},
\end{align*}
which implies \eqref{lem:int2}.
\end{proof}

\begin{lemma}[Brezis-Gallouet-Wainger inequality]\label{lem:BGW}
	Let $\nu\in (0,1)$.
	For $f\in H^1$, we have
	\[ \| f\|_{L^\infty}\les_\nu \| f\|_{H^{\frac12}}\log^{\frac12}\big( e+\| \p_xf\|_{L^2}\big) +\| f\|_{H^{\frac12}}^{1-\nu}.\]
\end{lemma}

\begin{proof}
	Applying Cauchy-Schwarz on the Fourier side, we have
	\[ \| f\|_{L^\infty}\les \| \wh{f}\|_{L^1}\les \big( \log^{\frac12} (1+R)\big) \| \bra{\xi}^{\frac12}\wh{f}\|_{L^2(|\xi|\leq R)}+R^{-\frac12}\| \xi \wh{f}\|_{L^2(|\xi|>R)}\]
	for any $R>0$.
	Setting $R:=\| \p_xf\|_{L^2}^2/\| f\|_{H^{\frac12}}^2$, we have
	\begin{align*}
		\| f\|_{L^\infty}&\les \Big\{ 1+\log^{\frac12}\Big( 1+\frac{\| \p_xf\|_{L^2}^2}{\| f\|_{H^{\frac12}}^2}\Big) \Big\} \| f\|_{H^{\frac12}}\\
		&\les \Big\{ \log^{\frac12}\big( e+\| \p_xf\|_{L^2}\big) +\log^{\frac12} \Big( 1+\frac{1}{\| f\|_{H^{\frac12}}}\Big) \Big\} \| f\|_{H^{\frac12}} \\
		&\les _\nu \| f\|_{H^{\frac12}}\log^{\frac12}\big( e+\| \p_xf\|_{L^2}\big) +\| f\|_{H^{\frac12}}^{1-\nu}.
	\end{align*}
\end{proof}


\section{A priori bounds}
\label{sec:apriori}

In this section, we establish various a priori estimates for smooth solutions $u^\ve$ of \eqref{kdnls-ve} that will be used in the proof of Theorem~\ref{thm:lwp}.
In particular, we need a priori bounds in higher regularities to prove persistence of regularity property.
Our proof is based on the energy estimates.
Throughout this section, we fix $\ve\in(0,1)$ and drop the superscript $\ve$.

\subsection{A priori bounds in $H^2$, $H^k$}

\begin{prop}[$H^2$ a priori bound]\label{prop:H2apriori}
There exist a continuous increasing function $C_0:\R_+\to [1,\infty)$ and a continuous decreasing function $T_0: \R_+\to (0,1]$ such that the following holds:

Let $\ve\in(0,1)$, $T>0$, and suppose that $u$ is a smooth solution to \eqref{kdnls-ve} on $[0,T]$ (say, $u\in C([0,T],H^{7})$) with an initial data $\phi$.
Then, we have
\begin{gather}
\| u(t)\|_{H^2}\leq C_0(\| \phi\|_{H^1})\| \phi\|_{H^2} \qquad \text{for}\quad 0\leq t\leq \min \{ T, T_0(\| \phi\|_{H^2})\} . \label{apriori-lwp1}
\end{gather}
\end{prop}
Note that all calculations in the following proof are justified for smooth solutions to \eqref{kdnls-ve}.
\begin{proof}
Using the equation \eqref{kdnls-ve} and interpolation inequalities, we see that
\begin{gather}\label{est:dtu0}
\frac{d}{dt}\| u(t)\|_{L^2}^2=-2\ve \| \p_xu(t)\|_{L^2}^2+2 {\rm Re} \int \beta \p_x\big[\cH (|u|^2)u\big]\,\ol{u}\,dx \les \| u(t)\|_{L^2}^2\| \p_xu(t)\|_{L^2}^2,
\end{gather}
and similarly,
\begin{gather}\label{est:dtu1}
\frac{d}{dt}\| \p_xu(t)\|_{L^2}^2\les \| u(t)\|_{L^2}^2\| \p_x^2u(t)\|_{L^2}^2.
\end{gather}
Let us move on to the second order case $\frac{d}{dt}\|\p_x^2u(t)\|_{L^2}^2$. 
Since the same type of estimate as \eqref{est:dtu1} gives a third order derivative which is not bounded by $\| u(t)\|_{H^2}$, we introduce the gauge function $e^{\rho^\pm[u]}$ defined in \eqref{def:gauge}.
By the relation 
\[ \p_x^2u=P_{\leq 1}\p_x^2u+e^{\rho^-[u]}e^{\rho^+[u]}Q_+\p_x^2u+e^{\rho^+[u]}e^{\rho^-[u]}Q_-\p_x^2u \]
and Lemma~\ref{lem:rho}, we have%
\footnote{From now on, the $\pm$ signs will occasionally be used to indicate the sum of these terms.}%
\begin{gather*}
\begin{aligned}
\| \p_x^2u \|_{L^2} &\leq \| P_{\leq 1}\p_x^2u \|_{L^2}+\| e^{\rho^\mp[u]}\|_{L^\infty}\| e^{\rho^\pm[u]}Q_\pm \p_x^2u\|_{L^2} \\
&\les e^{C\|u \|_{L^2}^2}\Big( \| e^{\rho^\pm[u]}Q_\pm \p_x^2u\|_{L^2} +\| u\|_{L^2}\Big) .
\end{aligned}
\end{gather*}
Instead of $\frac{d}{dt}\| \p_x^2u(t)\|_{L^2}^2$, we estimate the $t$-derivative of the right-hand side squared, which exceeds $\| u(t)\|_{H^2}^2$.
In practice, we define
\[ X(t):= C_1e^{C_1\|u(t)\|_{L^2}^2}\Big( \| e^{\rho^\pm[u(t)]}Q_\pm \p_x^2u(t)\|_{L^2}^2+\| u(t)\|_{L^2}^2 +C_1e^{C_1\|u(t)\|_{L^2}^2}\| u(t)\|_{H^1}^6\Big) \]
for some large $C_1>0$ so that $\| u(t)\|_{H^2}^2\leq \frac12X(t)$, and estimate $\frac{d}{dt}X(t)$.
Note that $X(t)\leq C_2(\| u(t)\|_{H^1}^2)\| u(t)\|_{H^2}^2$ for some smooth, continuous function $C_2:\R_+\to [1,\infty)$.

Then it remains to estimate $\frac{d}{dt}\| e^{\rho^\pm[u(t)]}Q_\pm \p_x^2u(t)\|_{L^2}^2$, and we will show that it is bounded by $C_3(\| u(t)\|_{H^2}^2)\| u(t)\|_{H^2}^2$ for some function $C_3(\cdot)$ after some modification (see \eqref{est:H2energy-X} below).
By the equation \eqref{kdnls-ve} and integration by parts, we have
\begin{align}
&\frac{d}{dt}\| e^{\rho^\pm[u(t)]}Q_\pm \p_x^2u(t)\|_{L^2}^2 \notag \\
&\quad =\int_{\R}\p_t\big( e^{2\rho^\pm[u]}\big) |Q_\pm \p_x^2u|^2\,dx +2{\rm Re}\int_{\R} e^{2\rho^\pm[u]}Q_\pm\p_x^2(\p_t u)\ol{Q_\pm \p_x^2u}\,dx \notag \\
&\quad =\int_{\R}\p_t\big( e^{2\rho^\pm[u]}\big) |Q_\pm \p_x^2u|^2\,dx -2{\rm Re}\,(\mp i\beta) \int_{\R} e^{2\rho^\pm[u]}|u|^2 Q_\pm\p_x^3u\ol{Q_\pm \p_x^2u}\,dx \notag \\
&\qquad +\ve \int \p_x^2\big( e^{2\rho^\pm[u]}\big) |Q_\pm \p_x^2 u|^2\,dx -2\ve \| e^{\rho^\pm[u]}Q_\pm \p_x^3u\|_{L^2}^2 +2{\rm Re}\int_{\R} e^{2\rho^\pm[u]}Q_\pm\p_x^3\mathcal{N}[u]\ol{Q_\pm \p_x^2u}\,dx \notag \\
&\quad \begin{aligned}
&\leq \Big( \| \p_te^{2\rho^\pm[u]}\|_{L^\infty}+\| \p_x^2e^{2\rho^\pm[u]}\|_{L^\infty}\Big) \| \p_x^2 u(t)\|_{L^2}^2-2\ve \| e^{\rho^\pm[u]}Q_\pm \p_x^3u\|_{L^2}^2 \\
&\quad +2{\rm Re}\int_{\R} e^{2\rho^\pm[u]}\Big( Q_\pm \p_x^3\mathcal{N}[u] -(\mp i\beta) |u|^2Q_\pm \p_x^3u\Big) \ol{Q_\pm \p_x^2u}\,dx.
\end{aligned} \label{eq:dt-vpm}
\end{align}
The first term of \eqref{eq:dt-vpm} can be estimated properly by Lemma~\ref{lem:rho}.
To treat the second line, we decompose the cubic terms as
\begin{align}
&Q_\pm \p_x^3\mathcal{N}[u] -(\mp i\beta) |u|^2Q_\pm \p_x^3u \notag \\
&\quad =Q_\pm \Big( 2\alpha |u|^2\p_x^3u+\beta \cH(|u|^2) \p_x^3u \Big) \label{H2-1}\\
&\qquad +\beta \Big( Q_\pm \big[ u\cH(\ol{u}\p_x^3u)\big] -u\ol{u}Q_\pm \cH \p_x^3u\Big) \label{H2-2}\\
&\qquad +Q_\pm \Big( \halbe\big( u\p_x^3\ol{u}\big)u\Big) \label{H2-3}\\
&\qquad + \text{(l.o.t.)} , \notag
\end{align}
where the lower order terms are treated by Lemma~\ref{lem:int} with $k=2$:
\[ \Big| 2{\rm Re}\int_{\R} e^{2\rho^\pm[u]}\text{(l.o.t.)}\ol{Q_\pm \p_x^2u}\,dx \Big| \les e^{C\| u\|_{L^2}^2}\| u\|_{H^2}^4. \]

For \eqref{H2-1}, we further decompose it and use integration by parts:
\begin{align*}
2{\rm Re}\int_{\R} e^{2\rho^\pm[u]}\eqref{H2-1} \ol{Q_\pm \p_x^2u}\,dx
&=2{\rm Re}\int_{\R} e^{2\rho^\pm[u]}\big[ Q_\pm ,2\alpha |u|^2+\beta \cH(|u|^2)\big] \p_x^3u\cdot \ol{Q_\pm \p_x^2u}\,dx \\
&\quad -\int_{\R}\p_x\Big[  e^{2\rho^\pm[u]}\big( 2\alpha |u|^2+\beta \cH(|u|^2)\big) \Big] |Q_\pm \p_x^2u|^2\,dx.
\end{align*}
The first commutator term is handled by Lemma \ref{lem:commutator}:
\[ \big\| \big[ Q_\pm, 2\alpha |u|^2+\beta \cH(|u|^2)\big] \p_x^3 u\big\|_{L^2}\les \| 2\alpha |u|^2+\beta \cH(|u|^2)\|_{H^2}\| u\|_{H^2}\les \| u\|_{H^2}^3.\]
The second term is easily treated.

For \eqref{H2-2}, it suffices to show%
\footnote{Here, we use the inhomogeneous decomposition.
Hence, we replace $P_N$ with $P_{\leq 1}$ for $N=1$.}%
\[ \sum_{N_0,N_1,N_2,N_3\geq 1}\Big\| P_{N_0}\Big[ Q_{\pm}\big( P_{N_1}u\cH(\ol{P_{N_2}u}\p_x^3P_{N_3}u)\big) -P_{N_1}u\ol{P_{N_2}u}Q_{\pm}\cH\p_x^3P_{N_3}u\Big] \Big\|_{L^2}\les \| u\|_{H^2}^3.\]
Note that the high-low interaction including problematic resonance (i.e., the case $N_3\gg N_1\vee N_2$) does not occur due to the cancellation effect.
Then, we may assume $N_3\les N_1\vee N_2\sim N_{\max}:=\max \{ N_0,N_1,N_2,N_3\}$, and we have
\begin{align*}
&\Big\| P_{N_0}\Big[ Q_{\pm}\big( P_{N_1}u\cH(\ol{P_{N_2}u}\p_x^3P_{N_3}u)\big) -P_{N_1}u\ol{P_{N_2}u}Q_{\pm}\cH\p_x^3P_{N_3}u\Big] \Big\|_{L^2}\\
&\quad \les \| \wh{P_{N_1}u}\|_{L^1}\| \wh{P_{N_2}u}\|_{L^1}\cdot N_3^3\| \wh{P_{N_3}u}\|_{L^2}\\
&\quad \les N_1^{\frac12}N_2^{\frac12}(N_1\vee N_2)N_3^{2}\| P_{N_1}u\|_{L^2}\| P_{N_2}u\|_{L^2}\| P_{N_3}u\|_{L^2}\\
&\quad \les N_{\max}^{-\frac12}N_1^2N_2^2N_3^2\| P_{N_1}u\|_{L^2}\| P_{N_2}u\|_{L^2}\| P_{N_3}u\|_{L^2},
\end{align*}
which is summable.

It remains to consider the contribution from \eqref{H2-3}.
We set $f(t):=e^{-it\p_x^2}u(t)$ and $h(t):=e^{-it\p_x^2}\big[ e^{2\rho^\pm[u(t)]}Q_\pm \p_x^2u(t)\big]$, then
\begin{align}
\begin{aligned}\label{eq:esti-r1}
&\int_{\R} e^{2\rho^\pm[u]}\eqref{H2-3} \ol{Q_\pm \p_x^2u}\,dx \\
&\quad =\int_{\R^3} e^{2it\eta\sigma}m(\xi,\sigma)(\xi-\eta-\sigma)^3 \wh{f}(\xi-\eta)\ol{\wh{f}(\xi-\eta-\sigma)}\wh{f}(\xi-\sigma)\ol{\wh{h}(\xi)}\,d\eta\, d\sigma\, d\xi ,
\end{aligned}
\end{align}
where
\[ m(\xi,\sigma):=c\,(\chi_\pm \varrho_{>1})(\xi)(\alpha -i\beta \sgn{\sigma}) .\]
We decompose the multiplier $(\xi-\eta-\sigma)^3$ using the relation 
\begin{equation}\label{xi-decomposition}
(\xi-\eta-\sigma)^2=(\xi-\eta-\sigma)(\xi-\eta)+(\xi-\eta-\sigma)(\xi-\sigma)-(\xi-\eta)(\xi-\sigma)+\eta\sigma .
\end{equation}
By Young's and H\"older's inequalities and Lemma~\ref{lem:rho}, the contribution from the first three terms is bounded by $e^{C\| u\|_{L^2}^2}\| u\|_{H^2}^4$.
Regarding the last case, by $\eta\sigma e^{2it\eta \sigma} =(2i)^{-1} \frac{d}{dt} e^{2it\eta \sigma} $, we apply integration by parts in $t$ to further divide into the following terms:
\begin{align}
&\frac{d}{dt}\int_{\R^3} e^{2it\eta\sigma}m(\xi,\sigma) (2i)^{-1}(\xi-\eta-\sigma) \wh{f}(\xi-\eta)\ol{\wh{f}(\xi-\eta-\sigma)}\wh{f}(\xi-\sigma)\ol{\wh{h}(\xi)}\,d\eta\, d\sigma\, d\xi,\label{eq:esti-r1-1}\\
&\int_{\R^3} e^{2it\eta\sigma}m(\xi,\sigma) (2i)^{-1} (\xi-\eta-\sigma) \p_t\Big[ \wh{f}(\xi-\eta)\ol{\wh{f}(\xi-\eta-\sigma)}\wh{f}(\xi-\sigma)\Big] \ol{\wh{h}(\xi)}\,d\eta\, d\sigma\, d\xi, \label{eq:esti-r1-2}\\
&\int_{\R^3} e^{2it\eta\sigma}m(\xi,\sigma) (2i)^{-1}(\xi-\eta-\sigma) \wh{f}(\xi-\eta)\ol{\wh{f}(\xi-\eta-\sigma)}\wh{f}(\xi-\sigma)\ol{\p_t\wh{h}(\xi)}\,d\eta\, d\sigma\, d\xi . \label{eq:esti-r1-3}
\end{align}
Among the contributions \eqref{eq:esti-r1-1}--\eqref{eq:esti-r1-3}, we put \eqref{eq:esti-r1-1} on hold for now and consider \eqref{eq:esti-r1-2} and \eqref{eq:esti-r1-3}. 
Using Lemmas~\ref{lem:rho} and \ref{lem:dt-bound}, \eqref{eq:esti-r1-2} is estimated by
\[ \| f\|_{H^1}^2\| \p_tf\|_{H^1}\| h\|_{L^2}\les e^{C\| u\|_{L^2}^2} \Big( \ve \| \p_x^3u\|_{L^2}\| u\|_{H^2}^3+\| u\|_{H^2}^4+\| u\|_{H^2}^6\Big) .\]
The term including the third-order derivative can be handled by the factor $-2\ve\| e^{\rho^\pm[u]}Q_\pm \p_x^3u\|_{L^2}^2$ in \eqref{eq:dt-vpm} from the fact that
\[ \| \p_x^3u\|_{L^2}\les \| u\|_{L^2}+e^{C\| u\|_{L^2}^2}\| e^{\rho^\pm[u]}Q_\pm \p_x^3u\|_{L^2}. \]
By these estimates, we prove the bound of \eqref{eq:esti-r1-2}:
\[ -\ve\| e^{\rho^\pm[u]}Q_\pm \p_x^3u\|_{L^2}^2+|\eqref{eq:esti-r1-2}|\les e^{C\| u\|_{L^2}^2}\Big( \| u\|_{H^2}^4+\| u\|_{H^2}^6\Big) .\]

For \eqref{eq:esti-r1-3}, we bound this term by
\[  \left\|  Q_{\pm} \big( \alpha u^2\p_x\ol{u}+\beta \cH(u\p_x\ol{u})u\big) \right\|_{H^1}\| \p_t h\|_{H^{-1}}. \]
On  one hand, we have
\[ \left\|Q_{\pm} \big( \alpha u^2\p_x\ol{u}+\beta \cH(u\p_x\ol{u})u\big)\right\|_{H^1}\les \| u\|_{H^2}^3.\]
On the other hand, we see that
\begin{align*}
	e^{it\p_x^2}\p_t  h&=(\p_t-i\p_x^2)(e^{2\rho^\pm[u]}Q_\pm \p_x^2 u) \\
	&=\big[ (\p_t-i\p_x^2)e^{2\rho^\pm[u]}\big] Q_\pm \p_x^2 u -2i\big[ \p_xe^{2\rho^\pm[u]}\big] Q_\pm \p_x^3 u +e^{2\rho^\pm[u]}Q_\pm \p_x^2 (\p_t -i\p_x^2) u.
\end{align*}
Using the estimate $\| \varphi \psi\|_{H^{-1}}\les \| \varphi\|_{W^{1,\infty}}\| \psi\|_{H^{-1}}$,%
\footnote{This is obtained as the dual estimate of $\| \varphi \psi\|_{H^1}\les \| \varphi\|_{W^{1,\infty}}\| \psi\|_{H^1}$.}
we have
\begin{align*}
	\| \p_t h\|_{H^{-1}}&\les \| (\p_t-i\p_x^2)e^{2\rho^\pm[u]}\|_{L^\infty}\| \p_x^2 u\|_{L^2}+\| \p_xe^{2\rho^\pm[u]}\|_{W^{1,\infty}}\| \p_x^3 u\|_{H^{-1}}\\
	&\quad +\| e^{2\rho^\pm[u]}\|_{W^{1,\infty}} \| \p_x^2 (\p_t-i\p_x^2)u \|_{H^{-1}}.
\end{align*}
From Lemmas~\ref{lem:rho} and \ref{lem:dt-bound}, we easily see that
\[ \| \p_t h\|_{H^{-1}}\les e^{C\| u\|_{L^2}^2}\Big\{ \ve \| \p_x^3u\|_{L^2}\big( 1+\| u\|_{H^2}^2\big) +\| u\|_{H^2}+\| u\|_{H^2}^5\Big\} .\]
Hence, the term \eqref{eq:esti-r1-3} is bounded as
\[ |\eqref{eq:esti-r1-3}|\les e^{C\| u\|_{L^2}^2}\Big\{ \ve \| \p_x^3u\|_{L^2}\big( \| u\|_{H^2}^3+\| u\|_{H^2}^5\big) +\| u\|_{H^2}^4+\| u\|_{H^2}^8\Big\} .\]
Note that we can control the third order term by using $-2\ve\| e^{\rho^\pm[u]}Q_\pm \p_x^3u\|_{L^2}^2$ in \eqref{eq:dt-vpm} as before.

Therefore, we obtain
\begin{align}\label{est:H2energy-X}
&\frac{d}{dt}\Big( \| e^{\rho^\pm[u(t)]}Q_\pm \p_x^2u(t)\|_{L^2}^2-\cI_2(t)   \Big) \leq C_3 (\| u(t)\|_{H^2}^2)\| u(t)\|_{H^2}^2
\end{align}
for some $C_3:\R_+\to [1,\infty)$, where
\begin{align*}
\cI_2(t):= 2{\rm Re} \int_{\R^3} e^{2it\eta\sigma}m(\xi,\sigma)(2i)^{-1}(\xi-\eta-\sigma) \wh{f}(t,\xi-\eta)\ol{\wh{f}(t,\xi-\eta-\sigma)}\wh{f}(t,\xi-\sigma)\ol{\wh{h}(t,\xi)}\,d\eta\, d\sigma\, d\xi .
\end{align*}
We observe that
\begin{align*}
|\cI_2(t)|&\sim \Big| 2{\rm Re} \int e^{\rho^\pm[u]}Q_{\pm} \big( \halbe (u\p_x\ol{u})u\big) \ol{e^{\rho^\pm[u]}Q_\pm \p_x^2u}\,dx\Big| \\
&\leq Ce^{C\| u\|_{L^2}^2}\| u\|_{H^1}^3\| e^{\rho^\pm[u]}Q_\pm \p_x^2u\|_{L^2}\leq \frac12 \| e^{\rho^\pm[u]}Q_\pm \p_x^2u\|_{L^2}^2+Ce^{C\| u\|_{L^2}^2}\| u\|_{H^1}^6.
\end{align*}
Then, for $C_1>0$ large enough, the functional 
\begin{align*}
\wt{X}(t)&:=X(t) - C_1 e^{C_1\| u(t)\|_{L^2}^2}\mathcal{I}_2(t)\\
&\;=C_1 e^{C_1\| u(t)\|_{L^2}^2}\Big( \| e^{\rho^\pm[u(t)]}Q_\pm \p_x^2u(t)\|_{L^2}^2-\cI_2(t) + \| u(t)\|_{L^2}^2+C_1 e^{C_1\| u(t)\|_{L^2}^2}\| u(t)\|_{H^1}^6\Big) 
\end{align*}
satisfies 
\[ \frac12X(t)\leq \wt{X}(t)\leq \frac32X(t) ,\qquad \| u(t)\|_{H^2}^2\leq \wt{X}(t)\leq \frac32C_2(\| u(t)\|_{H^1}^2)\| u(t)\|_{H^2}^2 \]
and, putting \eqref{est:dtu0}, \eqref{est:dtu1} and \eqref{est:H2energy-X} together, there exists a smooth, increasing function $C_4:\R_+\to [1,\infty)$ such that
\[ \frac{d}{dt}\wt{X}(t)\leq C_4 \big(\| u(t)\|_{H^2}^2 \big)\| u(t)\|_{H^2}^2\leq C_4\big( \wt{X}(t) \big)\wt{X}(t). \]
This implies that
\[ \wt{X}(t)\leq 2\wt{X}(0)\qquad \text{for}\quad 0\leq t\leq T^*,\]
where $T^*:= \min \left\{ T, \frac{\wt{X}(0)}{C_4(2\wt{X}(0))\cdot 2\wt{X}(0)}\right\}$. Therefore, the $H^2$ a priori bound \eqref{apriori-lwp1} follows by setting
\[ C_0(r):=\big[3C_2(r^2)\big]^{1/2},\qquad T_0(r):=\frac{1}{2C_4\big( 3C_2(r^2)r^2\big)}.
\]
\end{proof}

Next, we prove a priori estimate of $\| \p_x^ku(t)\|_{L^2}$ ($k\geq 3$) using the above $H^2$ bound.

\begin{prop}[Higher-order a priori bound]\label{prop:Hkapriori}
For any integer $k\geq 3$, there exists $C:\R_+\to [1,\infty)$ (depending on $k$) such that the following holds: Let $\ve\in (0,1)$, $T>0$, and suppose that $u$ is a smooth solution to \eqref{kdnls-ve} on $[0,T]$ (say, $u\in C([0,T],H^{k+5})$).
Then, we have
\begin{gather}
\| u(t)\|_{H^k}\leq C(\| \phi\|_{H^2}) \| \phi\|_{H^k}\qquad \text{for}\quad 0\leq t\leq \min \big\{ T,  T_0(\| \phi\|_{H^2}) \big\}, \label{apriori-lwp2}
\end{gather}
where $T_0(\cdot)$ is the function given in Proposition~\ref{prop:H2apriori}.
\end{prop}

\begin{proof}
The argument is parallel to the $H^2$ case in the previous proposition.
Noticing
\begin{equation}\label{est:dku}
 \| \p_x^ku\|_{L^2}\les e^{C\| u\|_{L^2}^2}\Big( \| e^{\rho^\pm[u]}Q_\pm \p_x^ku\|_{L^2}+\| u\|_{L^2}\Big), 
\end{equation}
we estimate the time derivative of $\| e^{\rho^\pm[u]}Q_\pm \p_x^ku\|_{L^2}$.
Similarly to \eqref{eq:dt-vpm}, we have
\begin{align*}
\frac{d}{dt}\| e^{\rho^\pm[u(t)]}Q_\pm \p_x^ku(t)\|_{L^2}^2 
&\leq \Big( \| \p_te^{2\rho^\pm[u]}\|_{L^\infty}+\| \p_x^2e^{2\rho^\pm[u]}\|_{L^\infty}\Big) \| \p_x^k u(t)\|_{L^2}^2-2\ve \| e^{\rho^\pm[u]}Q_\pm \p_x^{k+1}u\|_{L^2}^2 \\
&\quad +2{\rm Re}\int_{\R} e^{2\rho^\pm[u]}\Big( Q_\pm \p_x^{k+1}\mathcal{N}[u] -(\mp i\beta) |u|^2Q_\pm \p_x^{k+1}u\Big) \ol{Q_\pm \p_x^ku}\,dx.
\end{align*}
To estimate the last integral, we decompose the cubic terms as
\begin{align*}
Q_\pm \p_x^{k+1}\mathcal{N}[u] -(\mp i\beta) |u|^2Q_\pm \p_x^{k+1}u 
&=Q_\pm \Big( 2\alpha |u|^2\p_x^{k+1}u+\beta \cH(|u|^2) \p_x^{k+1}u \Big) \\ 
&\quad +\beta \Big( Q_\pm \big[ u\cH(\ol{u}\p_x^{k+1}u)\big] -u\ol{u}Q_\pm \cH \p_x^{k+1}u\Big) \\
&\quad +Q_\pm \Big( \halbe\big( u\p_x^{k+1}\ol{u}\big) u\Big) + \text{(l.o.t.)} ,
\end{align*}
and estimate the lower order terms by Lemma~\ref{lem:int}:
\[ \Big| 2{\rm Re}\int_{\R} e^{2\rho^\pm[u]}\text{(l.o.t.)}\ol{Q_\pm \p_x^ku}\,dx\Big| \les e^{C\| u\|_{L^2}^2}\| u\|_{H^2}^2\| u\|_{H^k}^2.\]
The first term on the right-hand side is treated by commutator estimate from Lemma~\ref{lem:commutator}:
\[ \big\| \big[ Q_\pm, 2\alpha |u|^2+\beta \cH(|u|^2)\big] \p_x^{k+1} u\big\|_{L^2}\les \| 2\alpha |u|^2+\beta \cH(|u|^2)\|_{H^2}\| u\|_{H^k}\les \| u\|_{H^2}^2\| u\|_{H^k},\]
and integration by parts.
The second term can be treated in the same way as \eqref{H2-2} in the $H^2$ case.
For the contribution from the third one, as before, we divide the term using \eqref{xi-decomposition} and apply integration by parts in $t$ to the last term, obtaining some terms bounded by $C(\| u\|_{H^2})\| u\|_{H^k}^2$ and
\begin{align*}
&\frac{d}{dt}\int_{\R^3} e^{2it\eta\sigma}m(\xi,\sigma) (2i)^{-1}(\xi-\eta-\sigma)^{k-1}\wh{f}(\xi-\eta)\ol{\wh{f}(\xi-\eta-\sigma)}\wh{f}(\xi-\sigma)\ol{\wh{h}(\xi)}\,d\eta\, d\sigma\, d\xi, \\
&\int_{\R^3} e^{2it\eta\sigma}m(\xi,\sigma) (2i)^{-1}(\xi-\eta-\sigma)^{k-1}\p_t\Big[ \wh{f}(\xi-\eta)\ol{\wh{f}(\xi-\eta-\sigma)}\wh{f}(\xi-\sigma)\Big] \ol{\wh{h}(\xi)}\,d\eta\, d\sigma\, d\xi, \\
&\int_{\R^3} e^{2it\eta\sigma}m(\xi,\sigma)(2i)^{-1} (\xi-\eta-\sigma)^{k-1}\wh{f}(\xi-\eta)\ol{\wh{f}(\xi-\eta-\sigma)}\wh{f}(\xi-\sigma)\ol{\p_t\wh{h}(\xi)}\,d\eta\, d\sigma\, d\xi ,
\end{align*}
where $f(t)=e^{-it\p_x^2}u(t)$ is the same as before but $h(t):=e^{-it\p_x^2}\big[ e^{2\rho^\pm[u(t)]}Q_\pm \p_x^ku(t)\big]$. 
The second integral is bounded by
\[ \Big( \| \p_tf\|_{L^2}\| \p_x^{k-1}f\|_{H^1}\| f\|_{H^1}+\| f\|_{H^1}^2\| \p_x^{k-1}\p_tf\|_{L^2}\Big) \| e^{2\rho^\pm[u]}\|_{L^\infty}\| \p_x^ku\|_{L^2},\]
while the third one is bounded by
\begin{align*}
&\| f\|_{H^1}^2\| \p_x^{k-1}f\|_{H^1}\| \p_th\|_{H^{-1}}\\
&\quad 
\les \| u\|_{H^1}^2\| u\|_{H^k}\Big\{ \begin{aligned}[t]
&\| (\p_t-i\p_x^2)e^{2\rho^\pm[u]}\|_{L^\infty}\| \p_x^ku\|_{L^2}+\| \p_xe^{2\rho^\pm[u]}\|_{W^{1,\infty}}\| \p_x^{k+1}u\|_{H^{-1}}\\
&+\| e^{2\rho^\pm[u]}\|_{W^{1,\infty}}\| \p_x^k(\p_t-i\p_x^2)u\|_{H^{-1}}\Big\} .
\end{aligned}
\end{align*}
These terms are treated by using Lemmas~\ref{lem:rho}, \ref{lem:dt-bound} and the factor $-2\ve\| e^{\rho^\pm[u]}Q_\pm \p_x^{k+1}u\|_{L^2}^2$ in the differential inequality.
We then obtain
\begin{align*}
&\frac{d}{dt}\Big( \| e^{\rho^\pm[u(t)]}Q_\pm \p_x^ku(t)\|_{L^2}^2-\cI_k(t) \Big) \leq C(\| u\|_{H^2})\| u\|_{H^k}^2,
\end{align*}
where
\begin{align*}
	\cI_k(t):= 2 {\rm Re} \int_{\R^3} e^{2it\eta\sigma}m(\xi,\sigma)(2i)^{-1}(\xi-\eta-\sigma)^{k-1}\wh{f}(\xi-\eta)\ol{\wh{f}(\xi-\eta-\sigma)}\wh{f}(\xi-\sigma)\ol{\wh{h}(\xi)}\,d\eta\, d\sigma\, d\xi.
\end{align*}
Using \eqref{est:dku} and interpolation inequalities, we see that
\begin{align*}
|\cI_k(t)|&\leq Ce^{C\| u\|_{L^2}^2}\| u\|_{H^1}^2\| \p_x^{k-1}u\|_{L^2}\| e^{\rho^\pm[u]}Q_\pm \p_x^ku\|_{L^2} \\
&\leq C(\| u\|_{H^1})\| u\|_{L^2}^{\frac{1}{k}}\| \p_x^ku\|_{L^2}^{\frac{k-1}{k}}\| e^{\rho^\pm[u]}Q_\pm \p_x^ku\|_{L^2}\\
&\leq C(\| u\|_{H^1})\Big\{ \| u\|_{L^2}^{\frac{1}{k}}\| e^{\rho^\pm[u]}Q_\pm \p_x^ku\|_{L^2}^{2-\frac1k}+\| u\|_{L^2}\| e^{\rho^\pm[u]}Q_\pm \p_x^ku\|_{L^2}\Big\} \\
&\leq \frac12 \| e^{\rho^\pm[u]}Q_\pm \p_x^ku\|_{L^2}^2+C(\| u\|_{H^1})\| u\|_{L^2}^2.
\end{align*}
We set
\begin{align*}
X_k(t)&:=Ce^{C\| u(t)\|_{L^2}^2}\Big( \| e^{\rho^\pm[u(t)]}Q_\pm \p_x^ku(t)\|_{L^2}^2- \cI_k(t)+C(\| u(t)\|_{H^1})\| u(t)\|_{L^2}^2\Big) .
\end{align*}
From \eqref{est:dku}, \eqref{est:dtu0} and \eqref{est:dtu1}, we have
\[ \| u(t)\|_{H^k}^2\leq X_k(t),\qquad \frac{d}{dt}X_k(t)\leq C(\| u(t)\|_{H^2})\Big( X_k(t)+\| u(t)\|_{H^2}^2\Big) .\]
By the $H^2$ bound \eqref{apriori-lwp1} obtained in Proposition~\ref{prop:H2apriori} and Gronwall's inequality, we have
\begin{align*}
\| u(t)\|_{H^k}^2&\leq \Big( X_k(0)+C(\| \phi\|_{H^2})\| \phi\|_{H^2}^2\Big) e^{C(\| \phi\|_{H^2})t}\leq C(\| \phi\|_{H^2})\| \phi\|_{H^k}^2
\end{align*}
for $ 0\leq   t\leq \min \big\{ T,T_0(\|\phi\|_{H^2}) \big\}$.
\end{proof}

\subsection{A priori bounds in $H^{1,1}$, $H^{k-1,1}$}

We turn to the estimate of the weighted norms.
Note that 
\[ \| f\|_{H^{k}}+\| f\|_{H^{k-1,1}}\sim \| f\|_{H^{k}}+\| J(t)f\|_{H^{k-1}} \]
for any integer $k\geq 1$, where the $t$-dependence of the implicit constant is uniform on each bounded interval.
Thus, we consider a priori estimates for $\| J(t)u(t)\|_{H^{k-1}}$ instead of $\| u(t)\|_{H^{k-1,1}}$.

\begin{prop}[$H^{1,1}$ a priori bound]\label{prop:H11apriori}
There exists $C:\R_+\to [1,\infty)$ such that the following holds: 
Let $\ve\in (0,1)$, $T>0$, and suppose that $u\in C([0,T],H^7\cap H^{6,1})$ is a smooth solution to \eqref{kdnls-ve} on $[0,T]$ with an initial data $\phi$. Then, we have
\begin{gather}
\| J(t)u(t)\|_{H^1}\leq C(\| \phi\|_{H^2})\| \phi\|_{H^{2}\cap H^{1,1}} \qquad \text{for}\quad 0\leq t\leq  \min \big\{ T, T_0(\| \phi\|_{H^2}) \big\}, \label{apriori-lwp3}
\end{gather}
where $T_0(\cdot)$ is the function given in Proposition~\ref{prop:H2apriori}.
\end{prop}

\begin{proof}
The equation \eqref{eq:Ju} for $J(t)u(t)$ yields that
\begin{align*}
\frac{d}{dt}\| Ju\|_{L^2}^2&\les \Big( \| \p_xJ\mathcal{N}[u]\|_{L^2}+\ve \| \p_xu\|_{L^2}+\| \mathcal{N}[u]\|_{L^2}\Big) \| Ju\|_{L^2} \\
&\leq C(\| u\|_{H^1})\Big( \| Ju\|_{H^1}^2+\| u\|_{H^1}^2\Big).
\end{align*} 
Then we estimate $\| \p_xJu\|_{L^2}$.
In view of 
\[ \| \p_xJu\|_{L^2}\les \| Ju\|_{L^2}+e^{C\| u\|_{L^2}^2}\| e^{\rho^\pm[u]}Q_\pm \p_xJu\|_{L^2},\] 
we may consider $\| e^{\rho^\pm[u]}Q_\pm \p_xJu\|_{L^2}$ and estimate 
\begin{align*}
	\frac{d}{dt}\| e^{\rho^\pm[u]}Q_\pm \p_xJu\|_{L^2}^2 &=\int_{\R} \p_t\big( e^{2\rho^\pm[u]}\big) |Q_\pm \p_xJu|^2\,dx +2{\rm Re}\int_{\R} e^{2\rho^\pm[u]}Q_\pm \p_x\p_t(Ju)\,\ol{Q_\pm \p_xJu}\,dx \\
&= \int_{\R} \big( \p_te^{2\rho^\pm[u]} +\ve \p_x^2e^{2\rho^\pm[u]}\big) |Q_\pm \p_xJu|^2\,dx -2\ve \| e^{\rho^\pm[u]}Q_\pm \p_x^2Ju\|_{L^2}^2 \\
&\quad -2{\rm Re}\,(\mp i\beta)\int_{\R}e^{2\rho^\pm[u]}|u|^2Q_\pm \p_x^2Ju\,\ol{Q_\pm\p_xJu}\,dx \\
&\quad +2{\rm Re}\int_{\R}e^{2\rho^\pm[u]}Q_\pm \Big( \p_x^2J\mathcal{N}[u]-2\ve \p_x^2u-\p_x\mathcal{N}[u]\Big) \ol{Q_\pm\p_xJu}\,dx. 
\end{align*}
From Lemmas~\ref{lem:rho} and \ref{lem:int}, the right-hand side is bounded by
\begin{align}
&C(\| u\|_{H^2})\Big( \| Ju\|_{H^1}^2+\| u\|_{H^2}^2\Big) -2\ve \| e^{\rho^\pm[u]}Q_\pm \p_x^2Ju\|_{L^2}^2 \label{eq:bound11-0} \\
&+ 2{\rm Re}\int_{\R}e^{2\rho^\pm[u]}Q_\pm \Big( \big( 2\alpha |u|^2+\beta \cH(|u|^2)\big) \p_x^2Ju \Big) \ol{Q_\pm\p_xJu}\,dx \label{eq:bound11-1} \\
&+ 2\beta {\rm Re}\int_{\R}e^{2\rho^\pm[u]}\Big( Q_\pm \big( u\cH(\ol{u} \p_x^2Ju) \big) -u\ol{u}Q_\pm \cH \p_x^2Ju \Big) \ol{Q_\pm\p_xJu}\,dx \label{eq:bound11-2} \\
&- 2{\rm Re}\int_{\R}e^{2\rho^\pm[u]}Q_\pm \Big( \halbe(u\p_x^2\ol{Ju})u\Big) \ol{Q_\pm\p_xJu}\,dx.\label{eq:bound11-3}
\end{align}
The arguments for \eqref{eq:bound11-1}--\eqref{eq:bound11-3} are analogous to the corresponding estimates for \eqref{H2-1}--\eqref{H2-3} in the proof of Proposition~\ref{prop:H2apriori}.
For \eqref{eq:bound11-1} and \eqref{eq:bound11-2}, we use Lemma~\ref{lem:commutator}, integration by parts, and dyadic decomposition suitably.
For \eqref{eq:bound11-3}, we can rewrite it as
\[ 2{\rm Re} \int_{\R^3} e^{2it\eta\sigma}m(\xi,\sigma)(\xi-\eta-\sigma)^2 \wh{f}(\xi-\eta)\ol{\wh{g}(\xi-\eta-\sigma)}\wh{f}(\xi-\sigma)\ol{\wh{h}(\xi)}\,d\eta\, d\sigma\, d\xi ,\]
where $m$, $f$ are the same as before and $g:=e^{-it\p_x^2}Ju$, and $h:=e^{-it\p_x^2}\big[ e^{2\rho^\pm[u]}Q_\pm \p_xJu\big]$. 
Then we decompose it by \eqref{xi-decomposition} and apply integration by parts in $t$ for the last term to obtain the following terms
\begin{align*}
&\frac{d}{dt}\int_{\R^3} e^{2it\eta\sigma}m(\xi,\sigma)(2i)^{-1}\wh{f}(\xi-\eta)\ol{\wh{g}(\xi-\eta-\sigma)}\wh{f}(\xi-\sigma)\ol{\wh{h}(\xi)}\,d\eta\, d\sigma\, d\xi ,\\
&\int_{\R^3} e^{2it\eta\sigma}m(\xi,\sigma)(2i)^{-1}\p_t\Big[ \wh{f}(\xi-\eta)\ol{\wh{g}(\xi-\eta-\sigma)}\wh{f}(\xi-\sigma)\ol{\wh{h}(\xi)}\Big] \,d\eta\, d\sigma\, d\xi .
\end{align*}
The second line can be estimated as before.
We only remark that substitution of the equation into $\p_tg$ and $\p_th$ creates the term $\ve \| \p_x^2Ju\|_{L^2}\| Ju\|_{H^1}$. 
This is treated using the second term in \eqref{eq:bound11-0}.

Thus we obtain the following differential inequality:
\begin{align*}
&\frac{d}{dt}\Big( \| e^{\rho^\pm[u(t)]}Q_\pm \p_xJ(t)u(t)\|_{L^2}^2- \cI_{1,1}(t)\Big) \leq C(\| u(t)\|_{H^2})\Big( \| J(t)u(t)\|_{H^1}^2+\| u(t)\|_{H^2}^2\Big),
\end{align*}
where
\begin{align*}
	\cI_{1,1}(t):= 2{\rm Re} \int_{\R^3} e^{2it\eta\sigma}m(\xi,\sigma)(2i)^{-1}\wh{f}(\xi-\eta)\ol{\wh{g}(\xi-\eta-\sigma)}\wh{f}(\xi-\sigma)\ol{\wh{h}(\xi)}\,d\eta\, d\sigma\, d\xi .
\end{align*}
The integral term can be estimated as
\begin{align*}
|\mathcal I_{1,1}(t)|&\leq Ce^{C\| u\|_{L^2}^2}\| u\|_{H^1}^2\| Ju\|_{L^2}\| e^{\rho^\pm[u]}Q_\pm \p_xJu\|_{L^2} \leq \frac12\| e^{\rho^\pm[u]}Q_\pm \p_xJu\|_{L^2}^2+C(\| u\|_{H^1})\| Ju\|_{L^2}^2.
\end{align*}
Now, we invoke the $H^2$ a priori estimate in Proposition~\ref{prop:H2apriori} to replace $\| u(t)\|_{H^2}$ and $\| u(t)\|_{H^1}$ by $C(\| \phi\|_{H^1})\| \phi\|_{H^2}$ for $0\leq t\leq  \min \big\{ T, T_0(\| \phi\|_{H^2}) \big\}$.
Then one gets
\begin{align*}
\frac{d}{dt}\| J(t)u(t)\|_{L^2}^2 &\leq C(\| \phi\|_{H^2})\Big( \| J(t)u(t)\|_{H^1}^2+\| \phi\|_{H^2}^2\Big) ,\\
\frac{d}{dt}\Big( \| e^{\rho^\pm[u(t)]}Q_\pm \p_xJ(t)u(t)\|_{L^2}^2 - \mathcal I_{1,1}(t)\Big) &\leq C(\| \phi \|_{H^2})\Big( \| J(t)u(t)\|_{H^1}^2+\| \phi \|_{H^2}^2\Big) ,\\
| \mathcal I_{1,1}(t) | &\leq \frac12\| e^{\rho^\pm[u(t)]}Q_\pm \p_xJ(t)u(t)\|_{L^2}^2+C(\| \phi\|_{H^2})\| J(t)u(t)\|_{L^2}^2.
\end{align*}
Hence, by defining
\begin{align*}
Y(t)&:=C(\|\phi\|_{H^2})\Big( \| e^{\rho^\pm[u(t)]}Q_\pm \p_xJ(t)u(t)\|_{L^2}^2+ C(\| \phi\|_{H^2})\| J(t)u(t)\|_{L^2}^2 \Big) ,\\
\wt{Y}(t)&:=C(\|\phi\|_{H^2})\Big( \| e^{\rho^\pm[u(t)]}Q_\pm \p_xJ(t)u(t)\|_{L^2}^2-\mathcal I_{1,1}(t)+ C(\| \phi\|_{H^2}) \| J(t)u(t)\|_{L^2}^2\Big) ,
\end{align*}
we have
\[ \| J(t)u(t)\|_{H^1}^2\leq \frac12Y(t)\leq \wt{Y}(t)\leq \frac32Y(t)\leq C(\| \phi\|_{H^2})\| J(t)u(t)\|_{H^1}^2 \]
and 
\[ \frac{d}{dt}\wt{Y}(t)\leq C(\| \phi\|_{H^2})\Big( \wt{Y}(t)+\| \phi\|_{H^2}^2\Big) \]
for $0\leq t\leq  \min \big\{ T, T_0(\| \phi\|_{H^2}) \big\}$.
By Gronwall's inequality, we obtain
\[ \| J(t)u(t)\|_{H^1}^2\leq \Big( \wt{Y}(0)+\| \phi\|_{H^2}^2\Big) e^{C(\| \phi\|_{H^2})T_0}\leq C(\| \phi\|_{H^2})\Big( \| x\phi\|_{H^1}^2+\| \phi\|_{H^2}^2\Big) .\]
This completes the proof of \eqref{apriori-lwp3}. 
\end{proof}

We now address a priori weighted estimate in higher regularity.

\begin{prop}[Higher-order a priori bound in weighted space]\label{prop:Hk1apriori}
For any integer $k\geq 3$, there exists $C:\R_+\to [1,\infty)$ (depending on $k$) such that the following holds.
Let $\ve\in (0,1)$, $T>0$, and suppose that $u$ is a smooth solution to \eqref{kdnls-ve} on $[0,T]$ (say, $u\in C([0,T],H^{k+5}\cap H^{k+4,1})$).
Then, we have
\begin{gather}
\| J(t)u(t)\|_{H^{k-1}}\leq C(\| \phi\|_{H^2\cap H^{1,1}})\| \phi\|_{H^{k}\cap H^{k-1,1}} \qquad \text{for}\quad 0\leq t\leq  \min \big\{ T, T_0(\| \phi\|_{H^2}) \big\}. \label{apriori-lwp4}
\end{gather}
\end{prop}

\begin{proof}
We follow the same strategy.
To estimate $\| \p_x^{k-1}Ju\|_{L^2}$, we may consider $\| e^{\rho^\pm[u]}Q_\pm \p_x^{k-1}Ju\|_{L^2}$, and we have
\begin{align*}
\frac{d}{dt}\| e^{\rho^\pm[u]}Q_\pm \p_x^{k-1}Ju\|_{L^2}^2&\leq C(\| u\|_{H^2},\| Ju\|_{H^1})\Big( \| Ju\|_{H^{k-1}}^2+\| u\|_{H^k}^2\Big) -2\ve \| e^{\rho^\pm[u]}Q_\pm \p_x^kJu\|_{L^2}^2 \\
&\quad + 2{\rm Re}\int_{\R}e^{2\rho^\pm[u]}Q_\pm \Big( \big( 2\alpha |u|^2+\beta \cH(|u|^2\big) \p_x^kJu \Big) \ol{Q_\pm\p_x^{k-1}Ju}\,dx \\
&\quad + 2\beta {\rm Re}\int_{\R}e^{2\rho^\pm[u]}\Big( Q_\pm \big( u\cH(\ol{u} \p_x^kJu)\big) -u\ol{u}Q_\pm \cH \p_x^kJu \Big) \ol{Q_\pm\p_x^{k-1}Ju}\,dx \\
&\quad - 2{\rm Re}\int_{\R}e^{2\rho^\pm[u]}Q_\pm \Big( \halbe(u\p_x^k\ol{Ju})u\Big) \ol{Q_\pm\p_x^{k-1}Ju}\,dx,
\end{align*}
where we used \eqref{lem:int2} in Lemma~\ref{lem:int}.
Here, we only see the part using the normal form. 
The last line is rewritten as
\begin{align*}
&2{\rm Re} \int_{\R^3} e^{2it\eta\sigma}m(\xi,\sigma)(\xi-\eta-\sigma)^k \wh{f}(\xi-\eta)\ol{\wh{g}(\xi-\eta-\sigma)}\wh{f}(\xi-\sigma)\ol{\wh{h}(\xi)}\,d\eta\, d\sigma\, d\xi 
\end{align*}
for $m$, $f$, $g$ as before and $h:=e^{-it\p_x^2}\big[ e^{2\rho^\pm[u]}Q_\pm \p_x^{k-1}Ju\big]$.
After decomposition \eqref{xi-decomposition} and integration by parts, we obtain
\begin{align*}
&\frac{d}{dt}\int_{\R^3} e^{2it\eta\sigma}m(\xi,\sigma)(2i)^{-1}(\xi-\eta-\sigma)^{k-2}\wh{f}(\xi-\eta)\ol{\wh{g}(\xi-\eta-\sigma)}\wh{f}(\xi-\sigma)\ol{\wh{h}(\xi)}\,d\eta\, d\sigma\, d\xi, \\
&\int_{\R^3} e^{2it\eta\sigma}m(\xi,\sigma)(2i)^{-1}(\xi-\eta-\sigma)^{k-2}\p_t\Big[ \wh{f}(\xi-\eta)\ol{\wh{g}(\xi-\eta-\sigma)}\wh{f}(\xi-\sigma)\ol{\wh{h}(\xi)}\,\Big] \,d\eta\, d\sigma\, d\xi .
\end{align*}
The bound of the latter can be obtained as before: we use the factor $-2\ve \| e^{\rho^\pm[u]}Q_\pm \p_x^kJu\|_{L^2}^2$ to absorb the terms with $\| \p_x^kJu\|_{L^2}$ appearing after substitution of the equation to $\p_tg$ and $\p_th$.
Thus we obtain
\begin{align*}
&\frac{d}{dt}\Big( \| e^{\rho^\pm[u]}Q_\pm \p_x^{k-1}Ju\|_{L^2}^2- \mathcal I_{k-1,1}(t)\Big) \leq C(\| u\|_{H^2},\| Ju\|_{H^1})\Big( \| Ju\|_{H^{k-1}}^2+\| u\|_{H^{k}}^2\Big) ,
\end{align*}
where
\begin{align*}
\mathcal I_{k-1,1}(t) := 2{\rm Re} \int_{\R^3} e^{2it\eta\sigma}m(\xi,\sigma)(2i)^{-1}(\xi-\eta-\sigma)^{k-2}\wh{f}(\xi-\eta)\ol{\wh{g}(\xi-\eta-\sigma)}\wh{f}(\xi-\sigma)\ol{\wh{h}(\xi)}\,d\eta\, d\sigma\, d\xi.
\end{align*}
By interpolation inequalities, the integral term is estimated as
\begin{align*}
|\mathcal I_{k-1,1}(t)|&\le Ce^{C\| u\|_{L^2}^2}\| u\|_{H^1}^2\| \p_x^{k-2}Ju\|_{L^2}\| e^{\rho^\pm[u]}Q_\pm \p_x^{k-1}Ju\|_{L^2} \\
&\leq \frac14\| e^{\rho^\pm[u]}Q_\pm \p_x^{k-1}Ju \|_{L^2}^2+C(\| u\|_{H^1})\| Ju\|_{L^2}^{\frac{2}{k-1}}\| \p_x^{k-1}Ju\|_{L^2}^{\frac{2(k-2)}{k-1}} \\
&\leq \frac12\| e^{\rho^\pm[u]}Q_\pm \p_x^{k-1}Ju \|_{L^2}^2+C(\| u\|_{H^1})\| Ju\|_{L^2}^2 .
\end{align*}
Then we apply $H^2$, $H^{k}$ and $H^{1,1}$ a priori estimates in Propositions~\ref{prop:H2apriori}--\ref{prop:H11apriori}:
\begin{gather*}
\frac{d}{dt}\Big( \| e^{\rho^\pm[u(t)]}Q_\pm \p_x^{k-1}J(t)u(t)\|_{L^2}^2-\mathcal I_{k-1,1}(t) \Big) \leq C(\| \phi\|_{H^2\cap H^{1,1}})\Big( \| J(t)u(t)\|_{H^{k-1}}^2+\| \phi\|_{H^{k}}^2\Big) ,
\end{gather*}
and
\begin{gather*}
| \mathcal I_{k-1,1}(t) | \leq \frac12\| e^{\rho^\pm[u(t)]}Q_\pm \p_x^{k-1}J(t)u(t) \|_{L^2}^2+C(\| \phi\|_{H^2})\| \phi\|_{H^2\cap H^{1,1}}^2.
\end{gather*}
Hence, 
\[ Y_k(t):=C(\| \phi\|_{H^2\cap H^{1,1}})\Big\{ \| e^{\rho^\pm[u(t)]}Q_\pm \p_x^{k-1}J(t)u(t)\|_{L^2}^2 -\mathcal I_{k-1,1}(t) + C(\| \phi \|_{H^2})\| \phi\|_{H^2\cap H^{1,1}}^2 \Big\} \]
satisfies
\[ \| J(t)u(t)\|_{H^{k-1}}^2\leq Y_k(t)\leq C(\| \phi\|_{H^2\cap H^{1,1}})\Big( \| J(t)u(t)\|_{H^{k-1}}^2+\| \phi\|_{H^2\cap H^{1,1}}^2\Big) \]
and 
\[ \frac{d}{dt}Y_k(t)\leq C(\| \phi\|_{H^2\cap H^{1,1}})\Big( Y_k(t)+\| \phi\|_{H^{k}}^2\Big) \]
for $0\leq t\leq \min \big\{ T, T_0(\| \phi\|_{H^2}) \big\}$.
By Gronwall's inequality, we obtain \eqref{apriori-lwp4}. 
\end{proof}


\section{A priori bounds for difference}
\label{sec:diff}

Finally, we establish a priori bounds of the difference of two solutions. 
We need the difference bounds for the uniqueness and continuous dependence.

\begin{prop}[$L^2$ a priori bound for difference]\label{prop:L2diff}
Let $0<\ve_2 \le \ve_1 <1$, and suppose that $u^{\ve_1},u^{\ve_2}\in C([0,T],H^5)$
be solutions to \eqref{kdnls-ve} on $[0,T]$ with the viscosity coefficient $\ve_1,\ve_2$ and the initial data $\phi^{\ve_1},\phi^{\ve_2}$, respectively.
Then, there exists $C>0$ depending on $M:=\max\limits_{j=1,2}\| u^{\ve_j}\|_{L^\infty_TH^2}$ and $T$ such that
\begin{equation}\label{apriori-diff1}
\| u^{\ve_1}-u^{\ve_2}\|_{L^\infty_TL^2}^2\leq C\Big( \| \phi^{\ve_1}-\phi^{\ve_2}\|_{L^2}^2+\ve_1^2\Big) .
\end{equation}
\end{prop}

Note that the constant $C$ in Proposition \ref{prop:L2diff} does not depend on $\ve_1,\ve_2$.

\begin{proof}
We begin with the decomposition
\begin{align*}
 	\| w\|_{L^2}&\leq \| P_{\leq N}w\|_{L^2}+\big\| e^{\rho^\mp [u^{\ve_2}]}\big\|_{L^\infty} \big\| e^{\rho^\pm[u^{\ve_2}]}P_\pm P_{>N}w \big\|_{L^2}\\
&\leq \| P_{\leq N}w\|_{L^2}+e^{C\| u^{\ve_2}\|_{L^2}^2}\big\| e^{\rho^\pm[u^{\ve_2}]}Q_{\pm N}w \big\|_{L^2},
\end{align*}
where $N\gg 1$ is a dyadic number to be chosen later and we write $Q_{\pm N}:=P_\pm P_{>N}$.
In particular, 
\[ \| w(t)\|_{L^2}^2\leq Z(t):= 2\| P_{\leq N}w(t)\|_{L^2}^2+C_1(M)\| e^{\rho^\pm[u^{\ve_2}(t)]}Q_{\pm N}w(t)\|_{L^2}^2 ,\qquad t\in [0,T].\]

Let us consider $\frac{d}{dt}Z(t)$.
The low frequency part is easily treated, since the derivative can be replaced with the factor $N$.
Using the equation \eqref{eq:w} for $w$, we have
\[ \frac{d}{dt}\| P_{\leq N}w(t)\|_{L^2}^2\leq \ve_1^2\| \p_x^2u^{\ve_1}(t)\|_{L^2}^2+C(M,N)\| w(t)\|_{L^2}^2\leq C(M,N)\Big( \| w(t)\|_{L^2}^2+\ve_1^2\Big) .\]
For the energy estimate of $\| e^{\rho^\pm[u^{\ve_2}(t)]}Q_{\pm N}w(t)\|_{L^2}^2$, we begin with
\begin{align}
&\frac{d}{dt}\| e^{\rho^\pm[u^{\ve_2}(t)]}Q_{\pm N}w(t)\|_{L^2}^2 \notag \\
&\quad =\int_{\R}\p_t\big( e^{2\rho^\pm[u^{\ve_2}]}\big) |Q_{\pm N}w|^2\,dx +2{\rm Re}\int_{\R} e^{2\rho^\pm[u^{\ve_2}]}Q_{\pm N}\p_t w\ol{Q_{\pm N}w}\,dx \notag \\
&\quad =\int_{\R}\p_t\big( e^{2\rho^\pm[u^{\ve_2}]}\big) |Q_{\pm N}w|^2\,dx -2{\rm Re}\,(\mp i\beta) \int_{\R} e^{2\rho^\pm[u^{\ve_2}]}|u^{\ve_2}|^2 Q_{\pm N}\p_x w\ol{Q_{\pm N}w}\,dx \notag \\
&\qquad + \ve_2 \int \p_x^2\big( e^{2\rho^\pm[u^{\ve_2}]}\big) |Q_{\pm N}w|^2\,dx -2\ve_2 \| e^{\rho^\pm[u^{\ve_2}]}Q_{\pm N}\p_xw \|_{L^2}^2 \notag \\
&\qquad + 2 (\ve_1-\ve_2) {\rm Re}\int_{\R} e^{2\rho^\pm[u^{\ve_2}]}Q_{\pm N}\p_x^2u^{\ve_1} \ol{Q_{\pm N}w}\,dx +2{\rm Re}\int_{\R} e^{2\rho^\pm[u^{\ve_2}]}Q_{\pm N}\p_x\wt{\mathcal{N}}[u^{\ve_1},w]\ol{Q_{\pm N}w}\,dx \notag \\
&\quad \begin{aligned} 
&\leq C(M)\Big( \| w(t)\|_{L^2}^2+\ve_1^2\Big) -2\ve_2 \| e^{\rho^\pm[u^{\ve_2}]}Q_{\pm N}\p_x w\|_{L^2}^2 \\
&\quad +2{\rm Re}\int_{\R} e^{2\rho^\pm[u^{\ve_2}]}\Big( Q_{\pm N}\p_x\wt{\mathcal{N}}[u^{\ve_1},w] -(\mp i\beta) |u^{\ve_2}|^2Q_{\pm N}\p_xw\Big) \ol{Q_{\pm N}w}\,dx.
\end{aligned} \label{eq:dt-wpm}
\end{align}
By \eqref{eq:wtN}, we decompose the cubic terms in \eqref{eq:dt-wpm} as
\begin{align*}
&Q_{\pm N}\p_x\wt{\mathcal{N}}[u^{\ve_1},w] -(\mp i\beta) |u^{\ve_2}|^2Q_{\pm N}\p_xw \\
&\quad =Q_{\pm N} \Big( \big( 2\alpha |u^{\ve_2}|^2+\beta \cH(|u^{\ve_2}|^2)\big) \p_x w\Big) \\
&\qquad +\beta \Big[ Q_{\pm N}\big[u^{\ve_2} \cH( \ol{u^{\ve_2}}\p_xw)\big] -u^{\ve_2}\ol{u^{\ve_2}}Q_{\pm N}\p_x\cH w \Big]\\
&\qquad +Q_{\pm N}\Big( \halbe (u^{\ve_2}\p_x\ol{w})u^{\ve_2}\Big) \\
&\qquad +Q_{\pm N}\wt{\cR}.
\end{align*}
The $L^2$ norm of the last term $Q_{\pm N}\wt{\cR}$ can be estimated by $C(M)\| w\|_{L^2}$.
The other terms can be treated similarly to \eqref{H2-1}--\eqref{H2-3} in the proof of Proposition~\ref{prop:H2apriori}.
We only consider the contribution from the third term,
\[ 2{\rm Re}\int_{\R}e^{2\rho^\pm[u^{\ve_2}]}Q_{\pm N}\Big( \halbe (u^{\ve_2}\p_x\ol{w}) u^{\ve_2}\Big) \ol{Q_{\pm N}w}\,dx, \]
for which we use the normal form approach. 
As before, we write the integral in the following form:
\begin{equation}\label{eq:dt-z-0}
\int_{\R^3} e^{2it\eta\sigma}m_N(\xi,\sigma)(\xi-\eta-\sigma)\wh{f}^{\ve_2}(\xi-\eta)\ol{\wh{g}(\xi-\eta-\sigma)}\wh{f}^{\ve_2}(\xi-\sigma)\ol{\wh{h}(\xi)}\,d\eta\, d\sigma\, d\xi ,
\end{equation}
where $f^{\ve_2}:=e^{-it\p_x^2}u^{\ve_2}$, $g:=e^{-it\p_x^2}w$, $h:=e^{-it\p_x^2}(e^{2\rho^\pm[u^{\ve_2}]}Q_{\pm N}w)$ and
\[ m_N(\xi,\sigma):=c\,(\chi_\pm\varrho_{>N})(\xi) (\alpha -i\beta \sgn{\sigma}) .\]
We then use the relation
\[ \xi (\xi-\eta-\sigma)=(\xi-\eta)(\xi-\sigma)-\eta\sigma \]
to decompose \eqref{eq:dt-z-0}.
The contribution from the first term is easily estimated; notice that $|\xi|\geq N$ in the support of $m_N$.
By integration by parts, the second term gives 
\begin{align}
&\frac{d}{dt}\int_{\R^3} e^{2it\eta\sigma}m_N(\xi,\sigma)(2i)^{-1}\xi^{-1} \wh{f}^{\ve_2}(\xi-\eta)\ol{\wh{g}(\xi-\eta-\sigma)}\wh{f}^{\ve_2}(\xi-\sigma)\ol{\wh{h}(\xi)}\,d\eta\, d\sigma\, d\xi,\label{eq:dt-z-1} \\
&\int_{\R^3} e^{2it\eta\sigma}m_N(\xi,\sigma)(2i)^{-1}\xi^{-1} \p_t\Big[ \wh{f}^{\ve_2}(\xi-\eta)\ol{\wh{g}(\xi-\eta-\sigma)}\wh{f}^{\ve_2}(\xi-\sigma) \Big] \ol{\wh{h}(\xi)}\,d\eta\, d\sigma\, d\xi, \label{eq:dt-z-2}\\
&\int_{\R^3} e^{2it\eta\sigma}m_N(\xi,\sigma)(2i)^{-1}\xi^{-1}\wh{f}^{\ve_2}(\xi-\eta)\ol{\wh{g}(\xi-\eta-\sigma)}\wh{f}^{\ve_2}(\xi-\sigma)\ol{\p_t\wh{h}(\xi)}\,d\eta\, d\sigma\, d\xi .\label{eq:dt-z-3}
\end{align}
For \eqref{eq:dt-z-2}, since we estimate $h$ in $L^2$, it is essentially enough to show 
\[ \big\| \p_t\big( (f^{\ve_2})^2\ol{g}\big)\big\|_{H^{-1}}\leq C(M)\Big( \| w\|_{L^2}+\ve_1\Big) .\]
Using the Sobolev embedding $L^1\hookrightarrow H^{-1}$ and the product estimate $\| \varphi\psi\|_{H^{-1}}\les \| \varphi\|_{H^1}\| \psi\|_{H^{-1}}$, we see that
\begin{align*}
\big\| \p_t\big( (f^{\ve_2})^2\ol{g}\big)\big\|_{H^{-1}}&\les \| f^{\ve_2}(\p_tf^{\ve_2})\ol{g}\|_{L^1}+\| (f^{\ve_2})^2\ol{\p_tg}\|_{H^{-1}}\\
&\les \| u^{\ve_2}\|_{H^1}\| (\p_t-i\p_x^2)u^{\ve_2}\|_{L^2}\| w\|_{L^2}+\| u^{\ve_2}\|_{H^1}^2\| (\p_t-i\p_x^2)w\|_{H^{-1}}.
\end{align*}
It is now clear from the equations for $u^{\ve_2}$ and $w$ that these terms are bounded appropriately.
Finally, for \eqref{eq:dt-z-3} we readily show that
\[ \| \p_th\|_{H^{-1}}\leq C(M)\Big( \| w\|_{L^2}+\ve_1\Big) .\]
In fact, the left-hand side is bounded by
\[ \| (\p_t-i\p_x^2)e^{2\rho^\pm[u^{\ve_2}]}\|_{L^\infty}\| w\|_{L^2}+\| \p_xe^{2\rho^\pm[u^{\ve_2}]}\|_{W^{1,\infty}}\| \p_xw_\pm\|_{H^{-1}} +\| e^{2\rho^\pm[u^{\ve_2}]}\|_{W^{1,\infty}}\| (\p_t-i\p_x^2)w\|_{H^{-1}}. \]
Hence, taking \eqref{eq:dt-z-1} into account, we obtain
\[ \frac{d}{dt}\Big( \| e^{\rho^\pm[u^{\ve_2}(t)]}Q_{\pm N}w(t)\|_{L^2}^2 -\mathcal{J}_0(t)\Big) \leq C(M)\Big( \| w(t)\|_{L^2}^2+\ve_1^2\Big) ,\]
where
\begin{align*}
	\mathcal{J}_0(t):=2{\rm Re} \int_{\R^3} e^{2it\eta\sigma}m_N(\xi,\sigma)(2i)^{-1}\xi^{-1} \wh{f}^{\ve_2}(\xi-\eta)\ol{\wh{g}(\xi-\eta-\sigma)}\wh{f}^{\ve_2}(\xi-\sigma)\ol{\wh{h}(\xi)}\,d\eta\, d\sigma\, d\xi.
\end{align*}
Noticing that
\[ | \mathcal{J}_0(t) | \leq \frac{C_2(M)}{N}\| w(t)\|_{L^2}^2 \leq \frac{C_2(M)}{N}Z(t),\]
we set $N=N(C_1(M), C_2(M))\gg 1$ so that 
\[ \frac12Z(t)\leq \wt{Z}(t):=Z(t)-C_1(M)\mathcal{J}_0(t)\leq \frac32Z(t).\]
By the above estimates and $\| w(t)\|_{L^2}^2\leq Z(t)\leq 2\wt{Z}(t)$, we have
\[ \frac{d}{dt}\wt{Z}(t)\leq C(M)\Big( \wt{Z}(t)+\ve_1^2\Big), \]
which implies 
\[ \| w(t)\|_{L^2}^2\leq 2\Big( \frac32Z(0)+\ve_1^2\Big) e^{C(M)T}\quad (t\in [0,T]).\]
This completes the proof of \eqref{apriori-diff1}.
\end{proof}

\begin{cor}\label{cor:L2diff}
Let $T>0$, and suppose that $u_1,u_2\in C([0,T],H^2)$ are solutions to \eqref{kdnls} on $[0,T]$ with the initial data $\phi_1,\phi_2$, respectively.
Then, there exists $C>0$ depending on $M:=\max\limits_{j=1,2}\| u_j\|_{L^\infty_TH^2}$ and $T$ such that
\[ \| u_1-u_2\|_{L^\infty_TL^2}\leq C\| \phi_1-\phi_2\|_{L^2} .\]
\end{cor}

\begin{proof}
In Proposition~\ref{prop:L2diff}, we have assumed the regularity $C_TH^5$ to rigorously justify the calculations. 
In fact, the argument may extend to general solutions in $C_TH^2$, since $u\in C_TH^2$ implies $\p_tu\in C_TL^2$ by the equation. 
The claim is shown by repeating the same argument but for the original equation \eqref{kdnls}. We omit the details.
\end{proof}

\begin{prop}[$H^2$ a priori bound for difference]\label{prop:H2diff}
Assume the hypotheses in Proposition~\ref{prop:L2diff}.  Let $u^{\ve_1},u^{\ve_2}\in C([0,T],H^9)$.
Then, there exists $C>0$ depending on $M_0:=\max\limits_{j=1,2}\| \phi^{\ve_j}\|_{H^2}$ but not on $\ve_1,\ve_2$ such that
\begin{equation}\label{apriori-diff2}
\| u^{\ve_1}(t)-u^{\ve_2}(t)\|_{H^2}^2\leq C\Big( \| \phi^{\ve_1}-\phi^{\ve_2}\|_{H^2}^2+\ve_1^2\big( 1+\| \phi^{\ve_1}\|_{H^4}^2+\| \phi^{\ve_1}\|_{H^3}^4\big) +\| \phi^{\ve_1}\|_{H^3}^4\| \phi^{\ve_1}-\phi^{\ve_2}\|_{L^2}^2 \Big) 
\end{equation}
for $0\leq t \leq  \min \big\{ T,T_0(M_0) \big\}$, where $T_0:\R_+\to (0,1]$ is the function given in Proposition~\ref{prop:H2apriori}.
\end{prop}

\begin{proof}
Proof follows a similar approach to that of Proposition~\ref{prop:L2diff}, with slight modifications.
We also note that we can use a priori bounds \eqref{apriori-lwp1} and \eqref{apriori-lwp2} up to $H^4$ as well as \eqref{apriori-diff1} under the assumption $u^{\ve_j}\in C([0,T],H^9)$.
We first estimate $\| w\|_{H^2}$ as 
\begin{align*}
\| w(t)\|_{H^2}^2&\leq C\Big( \| e^{\rho^\mp[u^{\ve_2}(t)]}\|_{L^\infty}^2\| e^{\rho^\pm[u^{\ve_2}(t)]}Q_\pm \p_x^2w(t)\|_{L^2}^2+\| w(t)\|_{L^2}^2\Big) \\
&\leq C(M_0) \Big( \| e^{\rho^\pm[u^{\ve_2}(t)]}Q_\pm \p_x^2w(t)\|_{L^2}^2+\|\phi^{\ve_1}-\phi^{\ve_2}\|_{L^2}^2+\ve_1^2 \Big). 
\end{align*}
Following the argument in \eqref{eq:dt-wpm}, we see that
\begin{align*}
\frac{d}{dt}\| e^{\rho^\pm[u^{\ve_2}(t)]}Q_\pm \p_x^2w(t)\|_{L^2}^2
&\leq C(M_0)\Big( \| \p_x^2w(t)\|_{L^2}^2+\ve_1^2\| \phi^{\ve_1}\|_{H^4}^2\Big) -2\ve_2\| e^{\rho^\pm[u^{\ve_2}(t)]}Q_\pm \p_x^3w(t)\|_{L^2}^2\\
&\quad +2{\rm Re}\int_{\R}e^{2\rho^\pm[u^{\ve_2}]}\Big( Q_\pm\p_x^3\wt{\mathcal{N}}[u^{\ve_1},w]-(\mp i\beta )|u^{\ve_2}|^2Q_\pm \p_x^3w\Big) \ol{Q_\pm\p_x^2w}\,dx.
\end{align*}
Among the cubic terms in $Q_\pm\p_x^3\wt{\mathcal{N}}[u^{\ve_1},w]-(\mp i\beta )|u^{\ve_2}|^2Q_\pm \p_x^3w$, we only consider those containing $\p_x^3w$.
Indeed, the $L^2$ norm of the other terms can be bounded by either
\[ \| (u^{\ve_1},w)\|_{H^2}^2\| w\|_{H^2}\leq C(M_0)\| w\|_{H^2}\]
or
\begin{align*}
\| \p_x^3u^{\ve_1}\|_{L^2}\| (u^{\ve_1},w)\|_{H^1}\| w\|_{H^1}&\leq C(M_0)\| \phi^{\ve_1}\|_{H^3}\| w\|_{L^2}^{\frac12}\| w\|_{H^2}^{\frac12}.
\end{align*}
The former is acceptable, while the latter gives the bound
\[ C(M_0)\| \phi^{\ve_1}\|_{H^3}\| w\|_{L^2}^{\frac12}\| w\|_{H^2}^{\frac32}\leq C(M_0)\Big( \| \phi^{\ve_1}\|_{H^3}^4\| w\|_{L^2}^2+\| w\|_{H^2}^2\Big) .\]

By a calculation similar to \eqref{eq:wtN}, the cubic terms with $\p_x^3w$ is written as
\begin{align*}
&Q_\pm \Big( \big( 2\alpha |u^{\ve_2}|^2+\beta \cH(|u^{\ve_2}|^2)\big) \p_x^3w\Big) +\beta \Big[ Q_\pm \big[u^{\ve_2} \cH( \ol{u^{\ve_2}}\p_x^3w)\big] -u^{\ve_2}\ol{u^{\ve_2}}Q_{\pm N}\p_x^3\cH w \Big] \\
&\quad +Q_\pm \Big( \halbe (u^{\ve_2}\p_x^3\ol{w})u^{\ve_2}\Big) ,
\end{align*}
and we again focus on the last term.
By the normal form approach, we get the terms
\begin{gather*}
\frac{d}{dt}\int_{\R^3} e^{2it\eta\sigma}m(\xi,\sigma)(2i)^{-1}(\xi-\eta-\sigma)\wh{f}^{\ve_2}(\xi-\eta)\ol{\wh{g}(\xi-\eta-\sigma)}\wh{f}^{\ve_2}(\xi-\sigma)\ol{\wh{h}(\xi)}\,d\eta\,d\sigma\,d\xi, \\
\int_{\R^3} e^{2it\eta\sigma}m(\xi,\sigma)(2i)^{-1}(\xi-\eta-\sigma) \p_t\Big[ \wh{f}^{\ve_2}(\xi-\eta)\ol{\wh{g}(\xi-\eta-\sigma)}\wh{f}^{\ve_2}(\xi-\sigma)\ol{\wh{h}(\xi)} \Big] \,d\eta\,d\sigma\,d\xi ,
\end{gather*}
where the notation is the same as in \eqref{eq:esti-r1}, \eqref{eq:dt-z-0} apart from $h:=e^{-it\p_x^2}(e^{\rho^\pm[u^{\ve_2}]}Q_\pm \p_x^2w)$.
The latter case is bounded by
\[ C(M_0)\Big( \| f^{\ve_2}\|_{H^1}\| \p_tf^{\ve_2}\|_{L^2}\| g\|_{H^2}\| h\|_{L^2}+\| f^{\ve_2}\|_{H^1}^2\| \p_tg\|_{H^1}\| h\|_{L^2}+\| f^{\ve_2}\|_{H^2}^2\| g\|_{H^2}\| \p_th\|_{H^{-1}}\Big) .\]
After estimating $\p_tg$ and $\p_th$, several third-order terms appear. 
Indeed, we see that
\begin{align*}
\| \p_tg\|_{H^1}&\leq \ve_2\| \p_x^3w\|_{L^2}+\ve_1\| \p_x^3u^{\ve_1}\| _{L^2}+C(M_0)\| w\|_{H^2},\\
\| \p_th\|_{H^{-1}}&\leq C(M_0)\Big( \ve_2\| \p_x^3w\|_{L^2}+\ve_1\| \p_x^3u^{\ve_1}\| _{L^2}+\| w\|_{H^2}\Big) .
\end{align*}
Inserting these estimates, we obtain the bound
\begin{align*}
&C(M_0)\Big( \| w\|_{H^2}^2+\ve_1\| \p_x^3u^{\ve_1}\|_{L^2}\| w\|_{H^2}+\ve_2\| \p_x^3w\|_{L^2}\| w\|_{H^2}\Big) \\
&\quad \leq C(M_0)\Big( \| w\|_{H^2}^2+\ve_1\| \p_x^3u^{\ve_1}\|_{L^2}\| w\|_{H^2}+\ve_2\| e^{\rho^\pm[u^{\ve_2}]}Q_\pm \p_x^3w\|_{L^2}\| w\|_{H^2}\Big) \\
&\quad \leq \ve_2\| e^{\rho^\pm[u^{\ve_2}]}Q_\pm \p_x^3w\|_{L^2}^2+C(M_0)\Big( \| w\|_{H^2}^2+\ve_1^2\| \phi^{\ve_1}\|_{H^3}^2\Big) .
\end{align*}
Note that the first term in the last line is absorbed into the viscosity term in the energy estimate.
We thus obtain
\begin{align*}
&\frac{d}{dt}\Big(\| e^{\rho^\pm[u^{\ve_2}(t)]}Q_\pm \p_x^2w(t)\|_{L^2}^2- \mathcal{J}_2(t)\Big) \leq C(M_0) \Big( \| w(t)\|_{H^2}^2+\ve_1^2\| \phi^{\ve_1}\|_{H^4}^2+\| \phi^{\ve_1}\|_{H^3}^4 \| w(t)\|_{L^2}^2\Big) ,
\end{align*}
where
\begin{align*}
\mathcal{J}_2(t):= 2{\rm Re} \int_{\R^3} e^{2it\eta\sigma}m(\xi,\sigma)(2i)^{-1}(\xi-\eta-\sigma) \wh{f}^{\ve_2}(\xi-\eta)\ol{\wh{g}(\xi-\eta-\sigma)}\wh{f}^{\ve_2}(\xi-\sigma)\ol{\wh{h}(\xi)}\,d\eta\, d\sigma\, d\xi.
\end{align*}
We  also estimate $\mathcal{J}_2(t)$, for $\mu>0$ arbitrarily small, as
\begin{align*}
|\mathcal{J}_2(t)|&\leq C(M_0)\| w\|_{H^1}\| w\|_{H^2}\leq C(M_0)\Big( \mu \| w\|_{H^2}^2+C_\mu \| w\|_{L^2}^2\Big) \\
&\leq C(M_0)\Big( \mu C(M_0)\| e^{\rho^\pm[u^{\ve_2}]}Q_\pm \p_x^2w\|_{L^2}^2+C_\mu \| w\|_{L^2}^2\Big) \\
&\leq \frac12 \| e^{\rho^\pm[u^{\ve_2}]}Q_\pm \p_x^2w\|_{L^2}^2+C(M_0)\| w\|_{L^2}^2.
\end{align*}
Hence, we set
\[ Z(t):=C(M_0)\Big( \| e^{\rho^\pm[u^{\ve_2}(t)]}Q_\pm \p_x^2w(t)\|_{L^2}^2-\mathcal{J}_2(t) +C(M_0)\big( \| \phi^{\ve_1}-\phi^{\ve_2}\|_{L^2}^2+\ve_1^2\big) \Big) ,\]
which satisfies
\[ \| w(t)\|_{H^2}^2\leq Z(t),\qquad Z(0)\leq C(M_0)\Big( \| \phi^{\ve_1}-\phi^{\ve_2}\|_{H^2}^2+\ve_1^2\Big) ,\]
and
\begin{align*}
	\frac{d}{dt}Z(t)&\leq C(M_0)\Big( \| w(t)\|_{H^2}^2+\ve_1^2\| \phi^{\ve_1}\|_{H^4}^2 +\| \phi^{\ve_1}\|_{H^3}^4 \big( \| \phi^{\ve_1}-\phi^{\ve_2}\|_{L^2}^2+\ve_1^2\big) \Big) \\
&\leq C(M_0)\Big( Z(t)+\ve_1^2\big( \| \phi^{\ve_1}\|_{H^4}^2+\| \phi^{\ve_1}\|_{H^3}^4\big) +\| \phi^{\ve_1}\|_{H^3}^4\| \phi^{\ve_1}-\phi^{\ve_2}\|_{L^2}^2\Big) .
\end{align*}
Therefore we obtain the desired bound by Gronwall's inequality.
\end{proof}

\begin{prop}[$H^{1,1}$ a priori bound for difference]\label{prop:H11diff}
Assume the hypotheses in Proposition~\ref{prop:L2diff}. Let $u^{\ve_1},u^{\ve_2}\in C([0,T],H^9\cap H^{8,1})$.
Then, there exists $C>0$ depending on 
$M_1:=\max\limits_{j=1,2}\| \phi^{\ve_j}\|_{H^2\cap H^{1,1}}$
but not on $\ve_1,\ve_2$ such that
\begin{align*}
\| J(t)\big[ u^{\ve_1}(t)-u^{\ve_2}(t)\big] \|_{H^1}^2\leq C\Big\{ \| \phi^{\ve_1}-\phi^{\ve_2}\|_{H^2\cap H^{1,1}}^2&+\ve_1^2\big( 1+\| \phi^{\ve_1}\|_{H^4\cap H^{3,1}}^2+\| \phi^{\ve_1}\|_{H^3\cap H^{2,1}}^4\big) \\
&+\| \phi^{\ve_1}\|_{H^3\cap H^{2,1}}^4\| \phi^{\ve_1}-\phi^{\ve_2}\|_{L^2}^2 \Big\} 
\end{align*}
for $0\leq t\leq \min \big\{ T,T_0(M_0) \big\}$.
\end{prop}

\begin{proof}
By the assumption $u^{\ve_j}\in C([0,T],H^9\cap H^{8,1})$, we can use a priori bounds obtained in previous propositions:
\eqref{apriori-lwp1} and \eqref{apriori-lwp3}, \eqref{apriori-lwp2} and \eqref{apriori-lwp4} up to $k=4$, \eqref{apriori-diff1} and \eqref{apriori-diff2}.
Since the strategy is exactly the same as the previous proposition, we skip the detailed argument and only give the resulting estimates.   
We begin with
\[ \| J(t)w(t)\|_{H^1}^2\les_{M_0}\| J(t)w(t)\|_{L^2}^2+\| e^{\rho^\pm[u^{\ve_2}(t)]}Q_\pm\p_xJ(t)w(t)\|_{L^2}^2,\]
and by the equation \eqref{eq:Jw} for $Jw$,
\[ \frac{d}{dt}\|J(t)w(t)\|_{L^2}^2\les_{M_1}\| J(t)w(t)\|_{H^1}^2+\| w(t)\|_{H^1}^2+\ve_1^2.\]
Next, we see that
\begin{align*}
&\frac{d}{dt}\| e^{\rho^\pm[u^{\ve_2}]}Q_\pm\p_xJw\|_{L^2}^2\\
&\quad \leq -2\ve_2\| e^{\rho^\pm[u^{\ve_2}]}Q_\pm\p_x^2Jw\|_{L^2}^2+C(M_1)\Big( \| Jw\|_{H^1}^2+\|w\|_{H^2}^2+\ve_1^2\|\phi^{\ve_1}\|_{H^4\cap H^{3,1}}^2+\| \phi^{\ve_1}\|_{H^3\cap H^{2,1}}^4\| w\|_{L^2}^2\Big) \\
&\qquad + 2{\rm Re}\int_{\R}e^{2\rho^\pm[u^{\ve_2}]}Q_\pm \Big( \big( 2\alpha |u^{\ve_2}|^2+\beta \cH(|u^{\ve_2}|^2\big) \p_x^2Jw \Big) \ol{Q_\pm\p_xJw}\,dx \\
&\qquad + 2\beta {\rm Re}\int_{\R}e^{2\rho^\pm[u^{\ve_2}]}\Big( Q_\pm \big( u^{\ve_2}\cH(\ol{u^{\ve_2}} \p_x^2Jw)\big) -u^{\ve_2}\ol{u^{\ve_2}}Q_\pm \cH \p_x^2Jw \Big) \ol{Q_\pm\p_xJw}\,dx \\
&\qquad - 2{\rm Re}\int_{\R}e^{2\rho^\pm[u^{\ve_2}]}Q_\pm \Big( \halbe(u^{\ve_2}\p_x^2\ol{Jw})u^{\ve_2}\Big) \ol{Q_\pm\p_xJw}\,dx,
\end{align*}
and the first two integrals on the right-hand side are estimated by $C(M_0)\| Jw\|_{H^1}^2$ using commutator estimate, integration by parts and dyadic decomposition.
We apply normal form to the last integral, which is then bounded by
\begin{align*}
C(M_0)\| Jw\|_{H^1}^2
&+\frac{d}{dt}\Big[ 2{\rm Re} \int_{\R^3} e^{2it\eta\sigma}m(\xi,\sigma)(2i)^{-1}\wh{f}^{\ve_2}(\xi-\eta)\ol{\wh{g}(\xi-\eta-\sigma)}\wh{f}^{\ve_2}(\xi-\sigma)\ol{\wh{h}(\xi)}\,d\eta\,d\sigma\,d\xi \Big] \\
&-2{\rm Re} \int_{\R^3} e^{2it\eta\sigma}m(\xi,\sigma)(2i)^{-1}\p_t\Big[ \wh{f}^{\ve_2}(\xi-\eta)\ol{\wh{g}(\xi-\eta-\sigma)}\wh{f}^{\ve_2}(\xi-\sigma)\ol{\wh{h}(\xi)}\Big] \,d\eta\,d\sigma\,d\xi,
\end{align*}
where $g:=e^{-it\p_x^2}Jw$ and $h:=e^{-it\p_x^2}(e^{2\rho^\pm[u^{\ve_2}]}Q_\pm \p_xJw)$.
The last integral is bounded by
\[ \ve_2\| e^{\rho^\pm[u^{\ve_2}]}Q_\pm\p_x^2Jw\|_{L^2}^2+C(M_1)\Big( \| Jw\|_{H^1}^2+\| w\|_{H^2}^2+\ve_1^2\| \phi^{\ve_1}\|_{H^3\cap H^{2,1}}^2\Big) ,\]
where we used the estimate
\[ \| \p_tg\|_{L^2}=\| (\p_t-i\p_x^2)Jw\|_{L^2}\les_{M_1} \ve_2\| e^{\rho^\pm[u^{\ve_2}]}Q_\pm\p_x^2Jw\|_{L^2}+\| Jw\|_{H^1}+\| w\|_{H^1}+\ve_1\| \phi^{\ve_1}\|_{H^3\cap H^{2,1}}.\]
Collecting the estimates above and applying the a priori difference estimates \eqref{apriori-diff1}, \eqref{apriori-diff2}, we see that
\begin{align*}
\frac{d}{dt}\Big( \| e^{\rho^\pm[u^{\ve_2}(t)]}Q_\pm\p_xJ(t)w(t)\|_{L^2}^2- \mathcal{J}_{1,1}(t)\Big) &\leq C(M_1)\Big( \| J(t)w(t)\|_{H^1}^2 + K_1 \Big) ,
\end{align*}
where
\[ \mathcal{J}_{1,1}(t):=2{\rm Re} \int_{\R^3} e^{2it\eta\sigma}m(\xi,\sigma)(2i)^{-1}\wh{f}^{\ve_2}(\xi-\eta)\ol{\wh{g}(\xi-\eta-\sigma)}\wh{f}^{\ve_2}(\xi-\sigma)\ol{\wh{h}(\xi)}\,d\eta\,d\sigma\,d\xi \]
and 
\begin{align*}
K_1:=\| \phi^{\ve_1}-\phi^{\ve_2}\|_{H^2}^2+\ve_1^2\big( 1+\| \phi^{\ve_1}\|_{H^4\cap H^{3,1}}^2+\| \phi^{\ve_1}\|_{H^3\cap H^{2,1}}^4\big) +\| \phi^{\ve_1}\|_{H^3\cap H^{2,1}}^4 \| \phi^{\ve_1}-\phi^{\ve_2}\|_{L^2}^2.
\end{align*}
We also see that
\begin{align*}
|\mathcal{J}_{1,1}(t)|&\leq C(M_0)\|Jw\|_{L^2}\| e^{\rho^\pm[u^{\ve_2}]}Q_\pm\p_xJw\|_{L^2}\leq \frac12\| e^{\rho^\pm[u^{\ve_2}]}Q_\pm\p_xJw\|_{L^2}^2+C(M_0)\| Jw\|_{L^2}^2.
\end{align*} 
Hence,
\[ Z(t):=C(M_1)\Big\{ \| e^{\rho^\pm[u^{\ve_2}(t)]}Q_\pm\p_xJ(t)w(t)\|_{L^2}^2 - \mathcal{J}_{1,1}(t)+C(M_0)\| J(t)w(t)\|_{L^2}^2 \Big\} \]
satisfies 
\[ \| J(t)w(t)\|_{H^1}^2\leq Z(t),\quad Z(0)\le C(M_1)\| \phi^{\ve_1}-\phi^{\ve_2}\|_{H^{1,1}}^2 \]
and, again by \eqref{apriori-diff2}, 
\[ \frac{d}{dt}Z(t)\leq C(M_1)\Big( Z(t) +K_1\Big) .\]
Therefore, Gronwall's inequality implies the desired bound.
\end{proof}


\section{Proof of global bound in $H^2$}
\label{sec:globalh2}

In this section we prove the global a priori bound in $H^2$ when $\beta<0$ in \eqref{kdnls}, which immediately leads us to Theorem \ref{thm:gwp}. 
To this end, we first recall the $H^1$ a priori bound in regular space.

\begin{prop}[Proposition~4.2 in \cite{KT-gauge}] \label{prop:H1bound}
	Let $\alpha ,\beta \in \R$ with $\beta <0$, and suppose that $u:[0,T]\times \R \to \mathbb{C}$ is a smooth solution to \eqref{kdnls} (say, $u\in C([0,T],H^{10})$).
	Then, there exists a constant $C>0$ depending on $\alpha$, $\beta$ and $\| \phi\|_{L^2}$ (not on $u$ itself and $T$) such that for any $t\in [0,T]$ we have
	\begin{gather*}
		\| u(t)\|_{L^2}^2+|\beta|\int_0^t\big\| D_x^{\frac12}(|u(s)|^2)\big\|_{L^2}^2\,ds = \| \phi\|_{L^2}^2,\\
		\| \p_xu(t)\|_{L^2}^2+\int_0^t\big\| D_x^{\frac32}(|u(s)|^2)\big\|_{L^2}^2\,ds \leq C \| \p_x\phi\|_{L^2}^2.
	\end{gather*}
\end{prop}
We shall prove the following $H^2$ bound:
\begin{prop}\label{prop:H2bound}
	Under the hypotheses of Proposition~\ref{prop:H1bound}, we have 
	\[ \| \p_x^2u(t)\|_{L^2}^2+\int_0^t\big\| D_x^{\frac52}(|u(s)|^2)\big\|_{L^2}^2\,ds \leq C_2 e^{C_1t^{4/3}}\]
	for any $t\in [0,T]$, where $C_k>0$ ($k=1,2$) are constants depending on $\alpha$, $\beta$ and $\| \phi\|_{H^k}$ (not on $u$ itself and $T$).
\end{prop}

By Propositions~\ref{prop:H1bound}, \ref{prop:H2bound} and the local well-posedness result in $H^2\cap H^{1,1}$ given by Theorem~\ref{thm:lwp}, we immediately prove the existence of global solutions $u\in C([0,\infty),H^2\cap H^{1,1})$ to \eqref{kdnls} for any initial data $\phi\in H^2\cap H^{1,1}$ in the dissipative case $\beta<0$, which implies Theorem~\ref{thm:gwp}.

\begin{rem}	
	The growth rate $e^{Ct^{4/3}}$ is certainly not optimal.  	We expect that the $H^2$ norm stays bounded.
\end{rem}


\begin{proof}[Proof of Proposition~\ref{prop:H2bound}]
If $u$ is a smooth solution to \eqref{kdnls}, then a direct calculation shows
\begin{align*}
\p_t(|\p_x^2 u|^2) &=2 {\rm Re} \Big[ \p_x^3\Big( i\p_x u+\cH_{\alpha,\beta}(|u|^2)u\Big) \ol{\p_x^2 u} \Big] \\
	&=\p_x\Big[ 2{\rm Re} \big( i\ol{\p_x^2 u}\p_x^3u\big) +\frac\alpha 2|\p_x^2u|^2-2\cH_{\alpha,\beta}(\p_x^2|u|^2)|\p_xu|^2+\cH_{\alpha,\beta}(|u|^2)|\p_x^2u|^2 \Big] \\
	&\qquad +\beta D_x(\p_x^2|u|^2)(\p_x^2|u|^2)+5\cH_{\alpha,\beta}(\p_x^2|u|^2)(\p_x|\p_xu|^2)+5\cH_{\alpha,\beta}(\p_x|u|^2)|\p_x^2u|^2,
\end{align*}
and hence,
\begin{align}\label{eq:uxx}
\begin{aligned}
&\frac{d}{dt}\int_\R |\p_x^2 u(t)|^2\,dx \\
&\quad =\beta \big\| D_x^{\frac52}(|u(t)|^2)\big\|_{L^2}^2+5\int_\R\cH_{\alpha,\beta}(\p_x^2|u|^2)\p_x (|\p_x u|^2)\,dx +5\int_\R \cH_{\alpha,\beta}(\p_x |u|^2)|\p_x^2 u|^2\,dx.
\end{aligned}
\end{align}
The second term of the RHS of \eqref{eq:uxx} is estimated by
\begin{align*}
&C\big\| \p_x^2 (|u(t)|^2)\big\|_{L^2}\| \p_x u(t)\|_{L^\infty}\| \p_x^2 u(t)\|_{L^2}\\
&\quad \leq C \big\| D_x^{\frac52}(|u(t)|^2)\big\|_{L^2}^{\frac12}\big\| D_x^{\frac32}(|u(t)|^2)\big\|_{L^2}^{\frac12}\| \p_x u(t)\|_{L^2}^{\frac12}\| \p_x^2 u(t)\|_{L^2}^{\frac32} \\
&\quad \leq \frac{|\beta|}{2}\big\| D_x^{\frac52}(|u(t)|^2)\big\|_{L^2}^2+C\big\| D_x^{\frac32}(|u(t)|^2)\big\|_{L^2}^{\frac23}\|\p_x u(t)\|_{L^2}^{\frac23}\|\p_x^2  u(t)\|_{L^2}^2,
\end{align*}
where the constant $C>0$ depends on $\alpha$ and $\beta$.
Using Lemma \ref{lem:BGW} with $f:=\cH_{\alpha,\beta}(\p_x|u|^2)$, the third term of the right-hand side of \eqref{eq:uxx} is estimated by
\begin{align*}
&\int_\R \cH_{\alpha,\beta}(\p_x |u|^2)|\p_x^2 u|^2\,dx \\
&\quad \leq C\|\cH_{\alpha,\beta}(\p_x |u|^2)\|_{L^\infty} \|\p_x^2 u\|_{L^2}^2\\
&\quad \leq C\Big( \big\| \p_x (|u(t)|^2)\big\|_{H^{\frac12}}\log^{\frac12}\big( e+\big\| \p_x^2 (|u(t)|^2)\big\|_{L^2}\big) +\big\| \p_x  (|u(t)|^2)\big\|_{H^{\frac12}}^{\frac23}\Big) \| \p_x^2 u(t)\|_{L^2}^2\\
&\quad \leq C(\| \phi\|_{L^2})\Big( \big\| \p_x (|u(t)|^2)\big\|_{H^{\frac12}}+\big\| \p_x (|u(t)|^2)\big\|_{H^{\frac12}}^{\frac23}\Big) \| \p_x^2 u(t)\|_{L^2}^2\log^{\frac12}\big( e+\| \p_x^2u(t)\|_{L^2}^2\big) ,
\end{align*}
where we have used $\| u(t)\|_{L^2}\leq \| \phi \|_{L^2}$ by Proposition~\ref{prop:H1bound}.
	Hence, we obtain
	\begin{align*}
		&\frac{d}{dt}\| \p_x^2 u(t)\|_{L^2}^2+\frac{|\beta|}{2}\big\| D_x^{\frac52}(|u(t)|^2)\big\|_{L^2}^2\\
		&\quad \leq C_1'\Big( \big\|\p_x (|u(t)|^2)\big\|_{H^{\frac12}}+\big\| \p_x (|u(t)|^2)\big\|_{H^{\frac12}}^{\frac23}\Big) \| \p_x^2 u(t)\|_{L^2}^2\log^{\frac12}\big( e+\| \p_x^2 u(t)\|_{L^2}^2\big) 
	\end{align*}
	for $t\in [0,T]$, where we have used Proposition~\ref{prop:H1bound} again and $C_1'>0$ depends on $\alpha$, $\beta$, $\| \phi\|_{H^1}$.
	Noticing the fact that
	\[ f(t)\geq 1,\quad f'(t)\leq g(t)f(t)\log^{\frac12}f(t)\quad \Longrightarrow \quad f(t)\leq \exp \Big[ \Big( \log^{\frac12}f(0)+\frac12\int_0^tg(s)\,ds\Big)^2\Big] ,\]
	and the local integrability
	\[ \int_0^t\Big( \big\| \p_x(|u(s)|^2)\big\|_{H^{\frac12}}+\big\| \p_x (|u(s)|^2)\big\|_{H^{\frac12}}^{\frac23}\Big) \,ds\les t^{\frac12}+t^{\frac23} \]
	by Proposition~\ref{prop:H1bound} and H\"older, we obtain the desired bound.
\end{proof}


\appendix

\section{Proof of local well-posedness}
\label{sec:proof-lwp}

In this section, we give a proof of Theorem~\ref{thm:lwp}. 
To this end, we introduce the following lemma.

\begin{lemma}\label{lem:lwp-ve}
Let $X^k$ be either $H^k$ or $H^k\cap H^{k-1,1}$ for any integer $k\geq 2$. 
For any $\ve\in (0,1)$, the regularized Cauchy problem \eqref{kdnls-ve} is locally well-posed in $X^k$.
The local existence time may be chosen as $T\sim \ve(1+\| \phi\|_{X^k})^{-4}$.
\end{lemma}

\begin{proof}
We give a proof only for the case $X^k=H^k\cap H^{k-1,1}$.
Let $\{ U_\ve(t):=e^{t(i\p_x^2+\ve \p_x^2)}\}_{t\geq 0}$ be the semigroup of the propagator for the linear equation $(\p_t-i\p_x^2-\ve \p_x^2)u=0$.
Define 
\[ \| u\|_{X_T}:=\sup _{t\in [0,T]}\Big( \| u(t)\|_{H^k}+\| J(t)u(t)\|_{H^{k-1}}\Big) ,\]
then $\| u\|_{X_T}\sim \sup\limits_{t\in [0,T]}\| u(t)\|_{H^k\cap H^{k-1,1}}$.
It suffices to prove that the map
\[ \Phi^\ve_\phi: u\mapsto U_\ve(t)\phi +\int_0^tU_\ve(t-s)\p_x\big[ \halbe(|u|^2)u\big] (s)\,ds \]
is a contraction on a complete metric space
\[ \Big\{ u\in C([0,T],H^k\cap H^{k-1,1})\,:\, \| u\|_{X_T}\leq C\big( \| \phi\|_{H^k}+\| x\phi\|_{H^{k-1}}\big) \Big\}\]
with the metric induced by $\| \cdot \|_{X_T}$.
Recalling $[J(t),\p_t-i\p_x^2]=0$ and $[J(t),\p_x^2]=-2\p_x$, we see that
\begin{align*}
&J(t)\Big[U_\ve (t)\phi +\int_0^tU_\ve (t-s)F(s)\,ds\Big] \\
&\quad =U_\ve (t)(x\phi) +\int_0^tU_\ve(t-s)\Big[ J(s)F(s)-2\ve \p_x\Big( U_\ve (s)\phi+\int_0^sU_\ve (s-r)F(r)\,dr\Big) \Big]\,ds\\
&\quad =U_\ve (t)\big[ x\phi-2\ve t\p_x\phi\big] +\int_0^tU_\ve(t-s)\big[ J(s)F(s)-2\ve (t-s)\p_xF(s)\big] \,ds.
\end{align*}
Using this and the Leibniz rule for $J(t)$, together with the parabolic smoothing estimate
\[ \Big\| \int_0^tU_\ve(t-s)\p_xF(s)\,ds\Big\|_{L^\infty_TH^k}\les \ve^{-\frac12}T^{\frac12}\| F\|_{L^\infty_TH^k}\]
and that $H^{k-1}$ is an algebra, we can easily show that $\Phi^\ve_\phi$ is indeed a contraction.
Note that the continuity of the map $t\mapsto \Phi^\ve_\phi[u](t)\in H^k\cap H^{k-1,1}$ is also verified by a standard argument. 
\end{proof}

\begin{proof}[Proof of Theorem~\ref{thm:lwp}]
(i) Let $\phi\in H^2$ be any initial data, and define $\phi^{\ve}:=P_{\leq \ve^{-\lambda}}\phi$ for $\ve\in (0,1)$, with any $\lambda\in (0,\frac12)$.
We have $\phi^\ve\in H^k$ for any $k\geq 2$ and 
\begin{alignat*}{2}
\| \phi^\ve\|_{H^k}&\les \ve^{-(k-2)\lambda}\| \phi^\ve\|_{H^2} &\qquad &(k>2),\\
\| \phi^{\ve_1}-\phi^{\ve_2}\|_{L^2}&\les \ve_1^{2\lambda}\| \phi^{\ve_1}-\phi^{\ve_2}\|_{H^2}&\qquad &(0<\ve_2<\ve_1<1).
\end{alignat*}
From these inequalities, it holds that
\begin{gather}\label{cond:lambda0}
\ve_1^2\big( 1+\| \phi^{\ve_1}\|_{H^4}^2+\| \phi^{\ve_1}\|_{H^3}^4\big) +\| \phi^{\ve_1}\|_{H^3}^4 \| \phi^{\ve_1}-\phi^{\ve_2}\|_{L^2}^2 \les _{\| \phi\|_{H^2}}\ve_1^{2\gamma}+\| \phi^{\ve_1}-\phi^{\ve_2}\|_{H^2}^2,
\end{gather}
where $\gamma :=\frac12 (2-4\lambda )>0$.
	
Let $u^\ve$ be a unique smooth solution of the regularized equation \eqref{kdnls-ve} with the parameter $\ve$ and the initial data $\phi^\ve$.
From Lemma~\ref{lem:lwp-ve} and the a priori estimate given in Proposition~\ref{prop:H2apriori}, the solution $u^\ve$ exists on a uniform-in-$\ve$ interval $[0,T_0]$ with $T_0=T_0(\| \phi\|_{H^2})>0$.
Then, by the difference estimate in Proposition~\ref{prop:H2diff} and \eqref{cond:lambda0}, we have
\begin{align*}
&\| u^{\ve_1}(t)-u^{\ve_2}(t)\|_{H^2}\leq C(\| \phi\|_{H^2})\Big( \| \phi^{\ve_1}-\phi^{\ve_2}\|_{H^2}+\ve_1^{\gamma}\Big)
\end{align*}
for any $0<\ve_2<\ve_1<1$.
Therefore, $\{ u^\ve\}_{\ve\in(0,1)}\subset C([0,T_0],H^2)$ is Cauchy as $\ve\to 0+$ and has a limit $u\in C([0,T_0],H^2)$ such that
\begin{gather}\label{est:cd0}
\| u^{\ve}-u\|_{L^\infty_{T_0}H^2}\leq C(\| \phi\|_{H^2})\Big( \| \phi^{\ve}-\phi\|_{H^2}+\ve^\gamma\Big) .
\end{gather}
Moreover, the limit $u$ is a solution to  \eqref{kdnls}.
From the $L^2$ Lipschitz estimate in Corollary~\ref{cor:L2diff}, the solution $u$ is unique in $C([0,T_0],H^2)$.
	
To show the continuous dependence on initial data, assume that a sequence $\{ \phi_j\}_j\subset H^2$ converges to $\phi$ in $H^2$.
We approximate $\phi_j$ in the same way, by $\phi_j^\ve:=P_{\leq \ve^{-\lambda}}\phi_j$, and let the corresponding solutions be denoted by $u_j,u^\ve_j$. 
Note that these solutions exist on the common interval $[0,T_0]$ corresponding to the data size $2\| \phi\|_{H^2}$ (for sufficiently large $j$).
Applying Proposition~\ref{prop:H2diff} (with $\ve_1=\ve_2=\ve$) and \eqref{cond:lambda0}, we see that
\[ \| u^{\ve}_j-u^\ve\|_{L^\infty_{T_0}H^2}
\leq C(\| \phi\|_{H^2})\Big\{ \| P_{\leq \ve^{-\lambda}}(\phi_j-\phi)\|_{H^2}+\ve^\gamma +\| P_{\leq \ve^{-\lambda}}(\phi_j-\phi)\|_{L^2}\ve^{-2\lambda}\Big\} . \]
Combining it with \eqref{est:cd0}, we have
\begin{align*}
\limsup_{j\to \infty} \| u_j-u\|_{L^\infty_{T_0}H^2}
&\les \limsup_{j\to \infty}\Big( \| P_{>\ve^{-\lambda}}\phi^j\|_{H^2}+\| P_{>\ve^{-\lambda}}\phi \|_{H^2}+\ve^\gamma+\| \phi_j-\phi\|_{H^2} +\| \phi_j-\phi\|_{L^2}\ve^{-2\lambda}\Big) \\
&\leq \sup_{j}\| P_{>\ve^{-\lambda}}\phi^j\|_{H^2}+\| P_{>\ve^{-\lambda}}\phi \|_{H^2}+\ve^\gamma .
\end{align*}
Since the last line tends to $0$ as $\ve\to 0$, this proves the convergence $u_j\to u$ in $C([0,T_0],H^2)$.


(ii) We first notice, for any $\phi\in H^2\cap H^{1,1}$ and $\phi^{\ve}:=P_{\leq \ve^{-\lambda}}\phi$, that $\| \phi^{\ve}\|_{H^2\cap H^{1,1}}\les \| \phi\|_{H^2\cap H^{1,1}}$ uniformly in $\ve\in(0,1)$ and $\phi^{\ve}\to \phi$ in $H^2\cap H^{1,1}$ as $\ve\to +0$.
This can be seen by $\mathcal{F}\big( [x,P_{\leq N}]\phi\big) (\xi)=i\varrho'_{\leq N}(\xi)\wh{\phi}(\xi)$ and $\| \varrho'_{\leq N}\|_{L^\infty}\les N^{-1}$.
Moreover, from the fact that $\varrho'_{\leq N}(\xi)$ is supported in $\{ |\xi|\sim N\}$, we have $\phi^\ve\in H^k\cap H^{k-1,1}$ for any $k\geq 2$ and 
\[ \| x\phi^\ve\|_{H^{k-1}}\les \ve^{-(k-2)\lambda}\big( \|x\phi^\ve\|_{H^1}+\ve^{\lambda}\| \phi^\ve\|_{H^1}\big) \qquad(k>2). \]
Using this, we see that
\begin{gather}\label{cond:lambda1}
\begin{aligned}
&\ve_1^2\big( 1+\| \phi^{\ve_1}\|_{H^4\cap H^{3,1}}^2+\| \phi^{\ve_1}\|_{H^3\cap H^{2,1}}^4\big) +\| \phi^{\ve_1}\|_{H^3\cap H^{2,1}}^4 \| \phi^{\ve_1}-\phi^{\ve_2}\|_{L^2}^2 \\
&\quad \les _{\| \phi\|_{H^2\cap H^{1,1}}}\ve_1^{2\gamma}+\| \phi^{\ve_1}-\phi^{\ve_2}\|_{H^2}^2.
\end{aligned}
\end{gather}
From Lemma~\ref{lem:lwp-ve} and the a priori estimate in Proposition~\ref{prop:H11apriori}, the solution $u^\ve$ to \eqref{kdnls-ve} belongs to $C([0,T_0],H^2\cap H^{1,1})$.
Then, the difference estimate in Proposition~\ref{prop:H11diff} and \eqref{cond:lambda1} imply that
\begin{align*}
&\| J(t)[u^{\ve_1}(t)-u^{\ve_2}(t)]\|_{H^1}\leq C(\| \phi\|_{H^2\cap H^{1,1}})\Big( \| \phi^{\ve_1}-\phi^{\ve_2}\|_{H^2\cap H^{1,1}}+\ve_1^{\gamma}\Big)
\end{align*}
for any $0<\ve_2<\ve_1<1$.
Hence, $\{ u^\ve\}_{\ve\in(0,1)}$ is also Cauchy in $C([0,T_0],H^2\cap H^{1,1})$, and the limit $u$ belongs to this space, satisfying
\begin{gather}\label{est:cd1}
\| J[u^{\ve}-u]\|_{L^\infty_{T_0}H^1}\leq C(\| \phi\|_{H^2\cap H^{1,1}})\Big( \| \phi^{\ve}-\phi\|_{H^2\cap H^{1,1}}+\ve^\gamma\Big) .
\end{gather}
To show the continuity of the solution map in $H^2\cap H^{1,1}$, it suffices to repeat the argument for the $H^2$ case in (i) with the additional estimates \eqref{est:cd1} and
\[ \| J[u^{\ve}_j-u^\ve]\|_{L^\infty_{T_0}H^1} \leq C(\| \phi\|_{H^2\cap H^{1,1}})\Big\{ \| P_{\leq \ve^{-\lambda}}(\phi_j-\phi)\|_{H^2\cap H^{1,1}}+\ve^\gamma +\| P_{\leq \ve^{-\lambda}}(\phi_j-\phi)\|_{L^2}\ve^{-2\lambda}\Big\} ,\]
which follows from Propositions~\ref{prop:H11diff} (with $\ve_1=\ve_2=\ve$) and \eqref{cond:lambda1}.

(iii) Suppose that $\phi\in X^k$ for some $k\geq 3$.
This means the approximating initial data $\{ \phi^\ve\}_{\ve\in(0,1)}$ are bounded in $X^k$.
Then, a priori bounds in Propositions~\ref{prop:Hkapriori} and \ref{prop:Hk1apriori} show that the solutions $\{ u^\ve\}_{\ve\in (0,1)}$ are bounded in $C([0,T_0],X^k)$.
Hence, for each $t\in [0,T_0]$, $u^{\ve}(t)$ converges weakly in $X^k$ along some sequence $\ve=\ve_j\to 0+$.
Since $u^\ve(t)$ converges to $u(t)$ in $X^2$, this implies $u(t)\in X^k$ and 
\[ \| u(t)\|_{X^k}\leq \sup _{\ve\in (0,1)}\| u^\ve\|_{L^\infty_{T_0}X^k}, \]
namely, $u\in L^\infty ([0,T_0],X^k)$.
By interpolation against $C([0,T_0],X^2)$, we obtain that $u\in C([0,T_0],X^{k-1})$.
\end{proof}


\section*{Acknowledgements}
The authors would like to appreciate Professor Yoshio Tsutsumi for having continued to provide deep insight into the study of KDNLS and other nonlinear dispersive equations.
They also thank the anonymous referees for their helpful suggestions.
N. Kishimoto was supported in part by KAKENHI Grant-in-Aid for Scientific Research (C) JP20K03678 funded by the Japan Society for the Promotion of Science(JSPS).
K. Lee was supported
in part by RS-2025-00514043 and  RS-2024-00463260,
the National Research Foundation of Korea(NRF) grant funded by the Korea government (MSIT) and (MOE), respectively.

\end{document}